\documentclass[12pt]{amsart}
\usepackage{amsmath}
\usepackage{amsthm}
\usepackage{amssymb}
\usepackage{enumerate}
\usepackage{xcolor}
\usepackage{comment}
\newtheorem{theorem}{Theorem}[section]
\newtheorem{lemma}[theorem]{Lemma}
\newtheorem{proposition}[theorem]{Proposition}
\newtheorem{corollary}[theorem]{Corollary}

\newtheorem{remar}[theorem]{Remark}
\theoremstyle{definition}
\newtheorem{example}[theorem]{Example}
\newtheorem{prob}[theorem]{Open Problem}
\newtheorem{definition}[theorem]{Definition}

\newcommand{\chara}{\mbox{\rm char}\,}

\newcommand{\cO}{\mathcal{O}}
\newcommand{\cM}{\mathcal{M}}

\newcommand{\R}{\mathbb R}
\newcommand{\Q}{\mathbb Q}

\newcommand{\Z}{\mathbb Z}

\begin{document}

\title[Almost Mathematics and deeply ramified fields]{Almost mathematics, K\"ahler differentials and deeply ramified fields}

\author{Steven Dale Cutkosky}

\dedicatory{Dedicated to Mel Hochster and Craig Huneke}
%

\thanks{partially supported by  NSF grant DMS 2054394.}

\address{Department of Mathematics, University of Missouri, Columbia,
MO 65211, USA}
\email{cutkoskys@missouri.edu}

\begin{abstract}
This article discusses ramification and the structure of relative K\"ahler differentials of extensions of valued fields. 
We  begin by  surveying  the theory developed in recent work with Franz-Viktor Kuhlmann and Anna Rzepka constructing the relative K\"ahler differentials of extensions of valuation rings in Artin-Schreier and Kummer extensions. We then show how this theory is applied  to give a  simple proof of Gabber and Ramero's characterization of deeply ramified fields.
Section \ref{Section2}  develops  the basics of almost mathematics, and should be accessible to a broad audience.
  Section \ref{Section3} gives a simple and self contained proof  of Gabber and Ramero's characterization of  when the extension of a rank 1 valuation of a field to its separable closure is weakly \'etale. In the final section, we  consider the equivalent conditions characterizing deeply ramified fields, as they are defined by Coates and Greenberg, and show that they are  the algebraic extensions $K$ of the p-adics $\Q_p$  which satisfy $\Omega_{\cO_{\overline{\Q_p}}|\cO_K}=0$,
for an algebraic closure $\overline{\Q_p}$ of $\Q_p$ which contains $K$.  
  
\end{abstract}

\maketitle

\section{Introduction} In this paper, we survey some  progress on understanding ramification and the structure of relative K\"ahler differentials of extensions of valued fields. 

Sections \ref{Section0} and \ref{Section1} survey our recent papers \cite{CKR} with Franz-Viktor Kuhlmann and Anna Rzepka and \cite{CK} with Franz-Viktor Kuhlmann. In Section \ref{Section0}, we outline  the theory developed in \cite{CKR} and \cite{CK} constructing the relative K\"ahler differentials of extensions of valuation rings in Artin-Schreier and Kummer extensions. In Section \ref{Section1}, we show how this theory is applied  to give a  simple proof (\cite[Theorem 1.2]{CK}) of the characterization of deeply ramified fields by Gabber Ramero, given in   \cite[Theorem 6.6.12]{GR}.
Section \ref{Section2}  develops  the basics of almost mathematics, and should be accessible to a broad audience.
 It is self contained, and does not require results from almost mathematics from other sources. This section only assumes that the reader is familiar with the basics of  commutative algebra as is developed in the books of  Matsumura \cite{Mat} or Eisenbud \cite{E}. Section \ref{Section3} gives a simple and self contained proof (via earlier results of this article) of \cite[Theorem 6.6.12]{GR}, characterizing when the extension of a rank 1 valuation ring to its separable closure is weakly \'etale. In Section \ref{Section4}, we  consider the equivalent conditions characterizing deeply ramified fields, as they are defined by Coates and Greenberg in \cite{CG}, and show that they are indeed the algebraic extensions $K$ of the p-adics $\Q_p$  which satisfy $\Omega_{\cO_{\overline{\Q_p}}|\cO_K}=0$,
for an algebraic closure $\overline{\Q_p}$ of $\Q_p$ which contains $K$.

Suppose that  $(K,v)$ is a valued field. We will denote the valuation ring of $v$ by $\mathcal O_K$, the maximal ideal of $\mathcal O_K$ by $\mathcal M_K$,
the residue field of $\mathcal O_K$ by $Kv$ and the value group of $v$ by $vK$. We denote an extension of valued fields $K\rightarrow L$ by $(L/K,v)$ where $v$ is the valuation on $L$ and we also denote its restriction to $K$ by $v$. Thus we have an induced extension of valuation rings $\cO_K\rightarrow \cO_L$. An extension of valued fields $(L/K,v)$ is unibranched if $v|K$ has a unique extension to $L$.

In Section \ref{Section0} of this article we survey some recent results with Franz-Viktor Kuhlmann and with Franz-Viktor Kuhlmann and Anna Rzepka computing the relative K\"ahler differentials of extensions of valuation rings in Artin-Schreier extensions and Kummer extensions of prime degrees.  We use this to compute the relative K\"ahler differentials of extensions of valuation rings in finite Galois extensions. We characterize when the K\"ahler differentials of such extensions are trivial.
Our results are valid for valuations of arbitrary rank. The problem of computing K\"ahler differentials in Galois extensions of degree $p$ is also studied in some papers by Thatte \cite{Th1}, \cite{Th2}. A  recent paper by Novacoski and Spivakovsky \cite{NoSp} computes the K\"ahler differentials of extensions of valued fields in terms of a generating sequence of the extension of valuations. 

A particularly interesting case of extensions are the  defect  extensions. A unibranched extension $K\rightarrow L$ of valued fields satisfies the inequality $$
[L:K]\ge (vL:vK)[Lv:Kv].
$$
 The extension 
has defect if $[L:K]>(vL:vK)[Lv:Kv]$. The defect of the extension is 
$$
\frac{[L:K]}{(vL:vK)[Lv:Kv]}
$$
which is a power of the characteristic $p$ of $Kv$ (Ostrowski's lemma, \cite[Theorem 2, page 236]{Ri}).
Defect can only occur in an extension if $Kv$ has positive  characteristic (Corollary to Theorem 25, page 78 \cite{ZS2}). Defect in Artin-Schreier extensions and Kummer extensions of prime degree are classified in \cite{Ku30} and \cite{KuRz} as being either independent or dependent. We show that this distinction is detected by the vanishing of the relative K\"ahler differentials. 

\begin{theorem}(Theorem 1.2 \cite{CKR})\label{Ocharind}
Take a valued field $(K,v)$ with $\chara Kv>0$; if $\chara K=0$, then assume that $K$
contains all $p$-th roots of unity. Further, take a Galois extension $(L|K,v)$ of 
prime degree with nontrivial defect. Then the
extension has independent defect if and only if 
\begin{equation}                         
\Omega_{\cO_L|\cO_K} \>=\> 0\>.
\end{equation}
\end{theorem}

 In Section \ref{Section1} of this paper, we give an outline of the application of our classification theorems of relative K\"ahler differentials to give a new proof (\cite[Theorem 1.2]{CK}) of the characterization of deeply ramified fields in \cite[Theorem 6.6.12]{GR}.
Before stating this theorem, we introduce some necessary notation.
A subgroup $\Delta$ of a value group $vK$ is called a convex subgroup if whenever an element $\alpha$ of $\Gamma$ belongs to $\Delta$ then all the elements $\beta$ of $vK$ which lie between $\alpha$ and $-\alpha$  also belong to $\Delta$. The set of all convex subgroups of $vK$ is totally ordered by the relation of inclusion. This concept is discussed on page 40 of \cite{ZS2}.
The completion $\hat K$ of a valued field generalizes the classical notion of completion of a discretely valued field. A definition and the basic properties of completions of valued fields for valuations of arbitrary rank can be found in Section 2.4 of \cite{EP}. 

A valued field which satisfies the equivalent conditions of the following theorem is called a deeply ramified field. 

\begin{theorem}(Theorem 6.6.12 \cite{GR})\label{GRThm2}
Let $(K,v)$ be a valued field, and identify $v$ with an extension of $v$ to a separable closure $K^{\rm sep}$ of $K$. Then the following two conditions are equivalent.
\begin{enumerate}
\item[1)] $\Omega_{\mathcal O_{K^{\rm sep}}|\mathcal O_K}=0$
\item[2)] Whenever $\Gamma_1\subsetneq\Gamma_2$ are convex subgroups of the value
group $vK$, then $\Gamma_2/\Gamma_1$ is not isomorphic to $\Z$. Further, if 
 $\chara Kv=p>0$, then the homomorphism
\begin{equation}                          \label{homOpO}
\cO_{\hat K}/p\cO_{\hat K} \ni x\mapsto x^p\in \cO_{\hat K}/p\cO_{\hat K}
\end{equation}
is surjective, where $\cO_{\hat K}$ denotes the valuation ring of the completion 
$\hat K$ of $(K,v)$. 
\end{enumerate}
\end{theorem}

We outline our  new proof (from Theorem 2.1 \cite{CK})   of this theorem  in Section \ref{Section1}. One feature of our proof is that our techniques are valid for valuations of all ranks. The original proof of Gabber and Ramero uses an induction argument to reduce to valuations of rank 1, where techniques of almost mathematics can be used.

In Section \ref{Section3} of this paper, we give a simplified proof of \cite[Proposition 6.6.2]{GR}.  

Deeply ramified fields were first defined by Coates and Greenberg in \cite{CG} for algebraic extensions of local fields. They give equivalent conditions characterizing these fields, which we state in Theorem \ref{CGtheorem}. The equivalent conditions they give are different from those of Theorem \ref{GRThm2}. Evidently, in the case of algebraic extensions of local fields, these two definitions of being deeply ramified agree. We show that this is so in Theorem \ref{localfieldDR}. One of the equivalent conditions of Coates and Greenberg is a vanishing statement in group cohomology: $H^1(K,\cM_k)=0$. Coates and Greenberg use this to prove other vanishing theorems on Abelian varieties $A$ over a deeply ramified field, from which they deduce results about the $p$-primary subgroups of the points $A(K)$ and $A(\overline \Q_p)$ of $A$. 

Basic setups are defined in Chapter 2 of \cite{GR} and at the beginning of Section \ref{Section2} of this paper, as are the definitions of almost zero modules and almost isomorphisms. 
Almost \'etale homomorphisms are defined in Chapter 3 of \cite{GR} and in Definition \ref{Def2} of this paper. Weakly \'etale homomorphisms are defined in Chapter 3 of \cite{GR} and in \ref{Def3} of this paper.  We give a simpler proof of the following theorem of \cite{GR} in Theorem \ref{Theorem4'} of Section \ref{Section3}.

\begin{theorem}(Proposition 6.6.2 \cite{GR})\label{GRThm1}.
Suppose that $(K,v)$ is a valued field, where $v$ is nondiscrete of rank 1. Let the basic setup be $(\cO_K,I)$ where $I=\cM_K$.  Identify $v$ with an extension of $v$ to a separable closure $K^{\rm sep}$ of $K$. Then the following are equivalent.
\begin{enumerate}
\item[1)] $\Omega_{\cO_{K^{\rm sep}}|\cO_K}$ is zero.
\item[2)] $\Omega_{\cO_{K^{\rm sep}}|\cO_K}$ is almost zero.
\item[3)] $\cO_{K}\rightarrow \cO_{K^{\rm sep}}$ is weakly \'etale.
\end{enumerate}
\end{theorem}

We state another result which we obtain in the course of our proof of Theorem \ref{GRThm1}, which sheds light on  almost \'etale extensions. This theorem is implicit in the proof in Proposition 6.6.2 of \cite{GR}. 

\begin{theorem}\label{Theorem3'} Suppose that $(L/K,v)$ is a finite separable unibranched extension of valued fields and that $v$ has rank 1 and is nondiscrete.  Let the basic setup be $(\cO_K,I)$ where $I=\cM_K$. Then $\cO_K\rightarrow \cO_L$ is almost finite \'etale if and only if $\Omega_{\cO_L|\cO_K}$ is almost zero.
\end{theorem}

The proofs given in \cite[Theorem 6.6.12 and Proposition 6.6.2]{GR} for the statements of  Theorem \ref{GRThm2} and Theorem \ref{GRThm1}   are very difficult. They are  in the later part of the book \cite{GR} and the proofs use much of the preceding material of the book.  The proofs which we give here and in \cite{CK} are simpler proofs of these results which are intended to be widely accessible.  
The statement and our proof of Theorem \ref{GRThm2} do not use almost mathematics. However, almost mathematics is required for the proof of Theorem \ref{GRThm1}, as the statement  of the proposition is itself in almost mathematics. 

In Section \ref{Section2} we develop the necessary almost mathematics for the proof of Theorem \ref{GRThm1}. This section is intended as a development of the basics of almost mathematics, and should be accessible to a broad audience. 

Almost mathematics was introduced by Faltings in \cite{Fa} and developed by Gabber and Ramero in a categorical framework in \cite{GR}. Faltings refers to the paper \cite{T} by Tate as an inspiration of his work, and in the paper \cite{CG} where deeply ramified fields are introduced, Coates and Greenberg also  refer to Tate's paper \cite{T} as a motivation for their work. The theorems in Section \ref{Section2} are those of \cite{GR}, but are stated and proved using ordinary mathematics. We do not use the category of almost mathematics which is defined and developed in \cite{GR}; all our statements and proofs are in terms of ordinary commutative algebra. The essential  ideas of the proofs of the main results of this section are  from \cite{GR} and \cite{Fa}.  We prove the results in Section \ref{Section2} in the full generality stated in \cite{GR}, even though this is not necessary for their application in Section \ref{Section3}.  In particular, our proofs in this section are valid for $R$-modules and $R$-algebras which have $R$-torsion. We mention a couple of recent papers, \cite{NS} by Nakazato and Shimomoto and \cite{IS}, by Ishiro and Shimomoto, which prove interesting results about almost mathematics, but are written in a way to be  readily understandable by algebraists.  Recently, many deep theorems in commutative algebra (especially the direct summand conjecture in mixed characteristic) have been proven using almost mathematics. The author mentions \cite{An1}, \cite{An2}, \cite{Bha} and \cite{CLMST}.

The author thanks the reviewer for suggestions improving the exposition of the paper.

\section{Relative K\"ahler differentials of extensions of valuation rings}\label{Section0} A valued field $(K,v)$ is a field $K$ with a valuation $v$; that is, there exists a totally ordered abelian group $vK$ called the value group of $v$ such that $v$ is a mapping  $v:K\setminus \{0\}\rightarrow vK$ satisfying $v(ab)=v(a)+v(b)$ and  $v(a+b)\ge\min\{v(a),v(b)\}$. We will set $v(0)=\infty$, where $\infty$ is larger than any element of $vK$.

The valuation ring $\mathcal O_K=\{f\in K\mid v(f)\ge 0\}$ of $v$ is a local ring with maximal ideal 
$\mathcal M_K=\{f\in K\mid v(f)>0\}$, and residue field $Kv=\mathcal O_K/\mathcal M_K$. The ring $\mathcal O_K$ is Noetherian if and only if $vK\cong \Z$ (Theorem 16, Chapter VI, Section 10, page 41 \cite{ZS2}).

A subgroup $\Delta$ of a value group $vK$ is called a convex subgroup if whenever an element $\alpha$ of $\Gamma$ belongs to $\Delta$ then all the elements $\beta$ of $vK$ which lie between $\alpha$ and $-\alpha$  also belong to $\Delta$. The set of all convex subgroups of $vK$ is totally ordered by the relation of inclusion. This concept is discussed on page 40 of \cite{ZS2}. 
In $(\R^n)_{\rm lex}$, $\R^n$ with the lexicographic order, the convex subgroups are
$$
0\subset e_n\R\subset e_{n-1}\R+e_{n}\R\subset \cdots\subset e_1\R+\cdots+e_n\R=(\R^n)_{\rm lex}.
$$
The convex subgroups $\Delta$ form a chain in $vK$. The prime ideals $P$ in $\mathcal O_K$ also form a chain.
The map $P\mapsto vK\setminus \pm\{v(f)\mid f\in P\}$ is a 1-1 correspondence from the prime ideals of $\mathcal O_K$ to the convex subgroups of $vK$ (Theorem 15, Chapter VI, Section 10, page 40 \cite{ZS2}). The rank of $v$, $\mbox{rank}(v)$ is the
order type of the chain of proper convex subgroups of $vK$.

 cardinality of convex subgroups of $vK$; 

If $K$ is an algebraic function field, then $v$ has finite rank, which implies that there exists some $n>0$ and an order preserving monomorphism $vK\rightarrow (\R^n)_{\rm lex}$. 

The theory of henselian fields is developed in Chapter 4 of \cite{EP}. A valued field $(K,v)$ is henselian if $v$ has a unique extension to every algebraic extension $L$ of $K$. Every valued field has a henselization. The henselization $K^h$ of $K$ is an extension of $K$ which is characterized (\cite[Theorem 5.2.2]{EP}) by the universal property that if $K\rightarrow K_1$ is an extension of valued fields such that $K_1$ is henselian, then there exists a unique homomorphism  $K^h\rightarrow K_1$, such that
$K\rightarrow K^h\rightarrow K_1$ is equal to the homomorphism $K\rightarrow K_1$. We have
\begin{equation}\label{eq40}
\mathcal O_{K^h}\cong (\mathcal O_K)^h
\end{equation}
where $\mathcal O_{K^h}$ is the valuation ring of the valued field $K^h$ and $(\mathcal O_K)^h$ is the henselization of the local ring $\mathcal O_K$ (by  \cite[Lemma 5.8]{CK}). By  \cite[Lemma 5.11]{CK}, if $(L/K,v)$ is a finite separable extension of valued fields, then
\begin{equation}\label{eq30}
\Omega_{\mathcal O_L^h|\mathcal O_K^h}\cong (\Omega_{\mathcal O_L|\mathcal O_K})\otimes_{\mathcal O_L}\mathcal O_{L^h},
\end{equation}
and so, 
$$
\Omega_{\mathcal O_L|\mathcal O_K}=0\mbox{ if and only if }\Omega_{\mathcal O_L^h|\mathcal O_K^h}=0
$$
since $\mathcal O_L\rightarrow \mathcal O_{L^h}$ is faithfully flat (by Lemma \ref{Lemma50}).

Recall that a Kummer extension $L/K$ is an extension $L=K[\vartheta]$ where $\vartheta^q-a=0$ for some $a\in K$,  the characteristic of $K$ does not divide $q$ and $K$ contains a primitive $q$-th root of unity. An Artin-Schreier extension is an extension $L=K[\vartheta]$ where $K$ has characteristic $p>0$ and $\vartheta^p-\vartheta-a=0$ for some $a\in K$.

In \cite{CKR} and \cite{CK}, we establish the following theorem by a case by case analysis of all types of ramification. 

\begin{theorem}\label{Theorem1*}(\cite[Theorem 1]{CK})
Suppose that $(K,v)$ is a valued field and $L$ is a Kummer extension of prime degree or an Artin-Schreier extension of $K$.  Then there is an explicit description of $\mathcal O_L$ as an $\mathcal O_K$-algebra and an explicit description of $\Omega_{\mathcal O_L|\mathcal O_K}$, giving a characterization of when $\Omega_{\mathcal O_L|\mathcal O_K}=0$. There are different  formulas  for different cases of invariants of the extension of valuations. Each case is of a completely different character and requires a different analysis.
\end{theorem}

We give the proof of the computation of $\Omega_{\mathcal O_L|\mathcal O_K}$, assuming the classification theorems of \cite{CK} and \cite{CKR}, and then discuss something of the method of proof of the classification theorems. In Section \ref{SectionN}, we will state explicitely what these theorems say in the special case of nondiscrete valuations of rank 1. These results will be used in the proof of Theorem \ref{Theorem4'}.

\begin{proof} 
Let $p$ be the characteristic of the residue
field $Kv$ and $q = [L : K]$ be a prime number. The description of $\Omega_{\cO_L|\cO_K}$  and the
characterization of vanishing of this module depend, among other information, on
the invariants of the valued field extension that appear in the following product:
$$
q = [L : K] = d(L|K) e(L|K) f(L|K) g(L|K)
$$
where $e(L|K) = (vL : vK)$, $f(L|K) = [Lv : Kv]$, $g(L|K)$ is the number of distinct
extensions of $v|K$ to $L$ and $d(L|K)$ is the defect of the extension, which is a power
of $p$. Since $q$ is a prime, exactly one of the factors will be equal to $q$, and the others
equal to 1. The description of $\Omega_{\cO_L|\cO_K}$ also depends on the rank and the structure
of the value group of $(K, v)$ if $d(L|K) = q$ or $e(L|K) = q$, and on whether $Lv|Kv$
is separable or inseparable if $f(L|K) = q$.

In the case of $d(L|K) = p$, The conclusions of Theorem \ref{Theorem1*} are proven in \cite[Theorems 4.2, 4.3 and 1.4]{CKR}. In the
case of $e(L|K) = q$, they are obtained in \cite[Theorem 4.6]{CK} for Artin-Schreier extensions
and \cite[Theorem 4.8]{CK} for Kummer extensions. If $f(L|K) = q$, then they are obtained
in \cite[Theorem 4.5]{CK} for Artin-Schreier extensions and in \cite[Proposition 5.6]{CK} (or \cite[Theorem 4.4]{CK})  and \cite[Theorem 4.7]{CK}
for Kummer extensions.
In the remaining case when $g(L|K) = q$, the extension $(L|K, v)$ is an inertial
extension. Thus $\Omega_{\cO_L|\cO_K}= 0$ by \cite[Proposition 5.6]{CK}.
\end{proof}

The case of Theorem \ref{Theorem1*} which contains the essential difference between  characteristic zero and positive characteristic $p$ of the residue field $Kv$ is when $K\rightarrow L$ is a defect extension of prime degree.
Defect is defined in the introduction of this paper. In this case, $[L:K]=p$, $vK=vL$ and $Kv=Lv$.
This is the case which destroys all attempts to resolve singularities in positive and mixed characteristic. 

Suppose that $K$ has characteristic $p>0$, and $K\rightarrow L$ is a defect Artin-Schreier extension. Let $\vartheta$ be a Artin-Schreier generator of $L/K$. We then have that 
$$
v(\vartheta-K)=\{v(\vartheta-c)\mid c\in K\}\subset vL_{<0}.
$$
The ramification ideal of $\mathcal O_L$, defined using Galois theory (\cite[Section 2.6]{CKR}), is 
$$
I_r=(a\in L\mid v(a)\in -v(\vartheta-K)).
$$
The defect is independent if $I_r^p=I_r$, or equivalently, 
 if $vK\setminus \pm v(\vartheta-K)$ is a convex subgroup of $vK$ (\cite[Theorems 1.2, 1.4]{CKR}).
 
 In the case of defect extensions, Theorem \ref{Theorem1*} is the following. 

\begin{theorem}\label{theorem2*} (\cite[Theorems 4.2 and 1.4]{CKR}) Suppose that $(L/K,v)$ is a defect Artin-Schreier extension. Then $\Omega_{\mathcal O_L|\mathcal O_K}\cong I_r/I_r^p$. Thus $\Omega_{\mathcal O_L|\mathcal O_K}=0$ if and only if $L/K$ is independent.
\end{theorem}

We give an idea of the proof in \cite{CKR}. For $\gamma\in vL=vK$, there exists $t_{\gamma}\in K$ such that $v(t_{\gamma})=-\gamma$. Let $\vartheta$ be an Artin-Schreier generator. We have that 
$$
\mathcal O_L=\cup_{c\in K}\mathcal O_K[\vartheta_c],
$$
where $\vartheta_c=t_{v(\vartheta-c)}(\vartheta-c)$. We also have that
$$
\Omega_{\mathcal O_L|\mathcal O_K}=\lim_{\rightarrow}\left((\Omega_{\mathcal O_K[\vartheta_c]|\mathcal O_K})\otimes_{\mathcal O_{K[\vartheta_c]}}\mathcal O_L\right).
$$
The proof of Theorem \ref{theorem2*} now follows from the computation of the K\"ahler differentials $\Omega_{\mathcal O_K[\vartheta_c]|\mathcal O_K}$ and the induced homomorphisms between them.

We may compute the K\"ahler differentials of an extension of valuation rings in a tower of   finite Galois extensions, using Equation (\ref{eq30}) and the following Proposition \ref{Prop3*} and Theorem \ref{Theorem4*}, to reduce the computation  to the case of Artin-Schreier and Kummer extensions. The K\"ahler differentials of the Artin-Schreier and Kummer extensions are then computed using Theorem \ref{Theorem1*}.

\begin{proposition}\label{Prop3*}(\cite[Proposition 4.1]{CK})
Suppose that $(L/K,v)$ is finite Galois and $(K,v)$ is henselian.  Then there exists a subextension $M$ of $K^{\rm sep}/L$ such that  $M/K$ is finite Galois and such that there is a factorization
$$
K\rightarrow (M|K)^{\rm in }=M_0\rightarrow M_1\rightarrow \cdots \rightarrow M_r=M
$$
where $(M|K)^{\rm in }$ is the inertia field of $M/K$ and  each $M_{i+1}/M_i$ is a unibranched Kummer extension of prime degree or an Artin-Schreier extension. 
\end{proposition}

The extension $(M|K)^{\rm in }/K$ is the largest subextension $N/K$ in $M/K$ for which $\mathcal O_K\rightarrow \mathcal O_N$ is \'etale local. We  have that 
\begin{equation}\label{eq42}
\Omega_{\cO_{(M|K)^{\rm in }}|\mathcal O_K}=0
\end{equation}
by  \cite[Theorem 5.5]{CK}.
The proof of Proposition \ref{Prop3*} is by Galois theory and ramification theory.

\begin{theorem}\label{Theorem4*}(\cite[Theorem 5.1]{CK}) Suppose that $(L/K,v)$ and $(M/L,v)$ are towers of finite Galois extensions of valued fields. Then
$$
0\rightarrow \Omega_{\mathcal O_L|\mathcal O_K}\otimes_{\mathcal O_L}\mathcal O_M\rightarrow \Omega_{\mathcal O_M|\mathcal O_K}\rightarrow \Omega_{\mathcal O_M|\mathcal O_L}\rightarrow 0
$$
is a short exact sequence of $\mathcal O_M$-modules. 
\end{theorem}

The first fundamental exact sequence,  \cite[Theorem 25.1]{Mat}, reduces the proof to establishing injectivity of the first map. As a consequence of Theorem \ref{Theorem4*}, we have that 
$$
\Omega_{\mathcal O_M|\mathcal O_K}=0\mbox{ if and only if }\Omega_{\mathcal O_M|\mathcal O_L}=0\mbox{ and } \Omega_{\mathcal O_L|\mathcal O_K}=0.
$$
This follows since an extension of valuation rings is faithfully flat (Lemma \ref{Lemma50}).

\section{A characterization of deeply ramified fields}\label{Section1} 
 Deeply ramified fields were introduced for algebraic extensions of $p$-adic fields by  Coates and Greenberg \cite{CG} for algebraic extensions of $p$-adic fields. They cite Tate's paper \cite{T} for inspiration for this concept. Gabber and Ramero generalize this  in \cite{GR}.
They prove the following theorem, restated from Theorem \ref{GRThm2} in the introduction.

\begin{theorem}(\cite[Theorem 6.6.12]{GR})\label{GRThm2*}
Let $(K,v)$ be a valued field, and identify $v$ with an extension of $v$ to a separable closure $K^{\rm sep}$ of $K$. Then the following two conditions are equivalent.
\begin{enumerate}
\item[1)] $\Omega_{\mathcal O_{K^{\rm sep}}|\mathcal O_K}=0$.
\item[2)] Whenever $\Gamma_1\subsetneq\Gamma_2$ are convex subgroups of the value
group $vK$, then $\Gamma_2/\Gamma_1$ is not isomorphic to $\Z$. Further, if 
 $\chara Kv=p>0$, then the homomorphism
\begin{equation}                          \label{homOpO*}
\cO_{\hat K}/p\cO_{\hat K} \ni x\mapsto x^p\in \cO_{\hat K}/p\cO_{\hat K}
\end{equation}
is surjective, where $\cO_{\hat K}$ denotes the valuation ring of the completion 
$\hat K$ of $(K,v)$. 
\end{enumerate}
\end{theorem}

Gabber and Ramero define a deeply ramified field to be a valued field which satisfies the equivalent conditions of Theorem \ref{GRThm2}. 

Convex subgroups and the completion $\hat K$ of  valued field are defined in the introduction.

The proof of Theorem \ref{GRThm2*} by Gabber and Ramero in \cite{GR} is a major application of their methods of almost mathematics. Their proof uses the derived cotangent complex in a serious way, and  uses  other sophisticated methods. 

We have that under the natural extension of $v$ to $\hat K$, that $v\hat K=vK$ and $\hat Kv=Kv$. Thus the second condition of 2) of Theorem \ref{GRThm2*} implies that if $K$ is a deeply ramified field such that $Kv$ has positive characteristic, then $Kv$ is perfect. 

A perfectoid field  (Definition 3.1 \cite{Sch}) is a complete valued field $(K,v)$ of residue characteristic $p>0$, where $v$ is a rank 1 non discrete valuation $v$, such that the Frobenius homomorphism is surjective on $\mathcal O_K/p\mathcal O_K$.

\begin{corollary} All perfectoid fields are deeply ramified.
\end{corollary}

We give a new proof of Theorem \ref{GRThm2*} in  \cite[Theorem 1.2]{CK}. Our proof is a direct proof for valuations of arbitrary rank. The proof in \cite{GR} is by induction on the rank of the valuation to reduce to the case of a rank 1 valuation where the techniques of almost mathematics are applicable. 
Here is a brief outline of the proof. By  \cite[Theorem 16.8]{E}, we have that
\begin{equation}\label{eq41}
\Omega_{\mathcal O_{K^{\rm sep}}|\mathcal O_K}=\lim_{\rightarrow}\left((\Omega_{\mathcal O_L|\mathcal O_K})\otimes_{\mathcal O_L}\mathcal O_K^{\rm sep}\right)
\end{equation}
where the limit is over the finite Galois subextensions $L/K$ of $K^{\rm sep}/K$. Equation (\ref{eq30}), Proposition \ref{Prop3*} and Theorem \ref{Theorem4*} reduce  computation of the vanishing of $\Omega_{\mathcal O_L|\mathcal O_K}$ in finite Galois extensions to our analysis of  Kummer and Artin-Schreier extensions in Theorem \ref{Theorem1*}, from which Theorem \ref{GRThm2*} is deduced.

Using the second equivalent condition of Theorem \ref{GRThm2*}, Kuhlmann and Rzepka showed the following striking theorem. 
\begin{theorem}\label{Thm6*}(\cite[Theorem 1.2 and Proposition 4.12]{KuRz}) Only independent defect can occur above a deeply ramified field. 
\end{theorem}
A different proof of this theorem follows from Theorem \ref{theorem2*} and Equation (\ref{eq30}), Proposition \ref{Prop3*} and Theorem \ref{Theorem4*} of this paper.

 We have the following characterization of deeply ramified fields when $Kv$ has positive characteristic $p>0$.
 
 \begin{theorem}\label{thm7*} (\cite[Theorem 1.3]{CK}) Let $(K,v)$ be a valued field of residue characteristic $p>0$. If $K$ has characteristic 0, assume in addition that it contains all $p$-th roots of unity.  Then $(K,v)$ is a deeply ramified field if and only if $\Omega_{\mathcal O_L|\mathcal O_K}=0$ for all Galois extensions $L/K$ of degree $p$.
 \end{theorem}

\section{Artin-Schreier and Kummer extensions of rank 1 non discrete valuations}\label{SectionN}

We will restrict now to the case when $v$ has rank 1 and is nondiscrete. The statements and proofs of the theorems cited from \cite{CKR} and \cite{CK} in the proof of Theorem \ref{Theorem1*} become simpler with this restriction.  This is the case of relevance in almost mathematics, and the statements of these theorems, which we give explicitly here with the restriction that $v$ has rank 1 and is not discrete,  will be used  in the proof of Theorem \ref{Theorem4'} later in this paper.  In this case, the theorems have relatively simple statements, as we are only dealing with valuation groups which are subgroups of the reals. In the general case, the structure of the valuation groups depends on all of the (possibly infinitely many) convex subgroups, and the statements of the theorems rely on the change of the convex subgroups under the extension.  We remark that the theorems in this section are exhaustive for extensions of rank 1 nondiscrete valuations in Artin-Schreier and Kummer extensions, by the proof of Theorem \ref{Theorem1*}.

In the following, we  state explicitely the conclusions of \cite[Theorems 4.2, 4.3 and 1.4]{CKR},  \cite[Theorem 4.5 - 4.8]{CK}) and \cite[Proposition 5.6]{CK}, with the restriction that $v$ is rank 1 and is nondiscrete.  

The consequence of \cite[Proposition 5.6]{CK} is the following. 
\begin{proposition}\label{PropN1} Suppose that  $(L/K,v)$ is finite Galois, 
$$
[L:K]=g(L|K)[Lv:Kv]
$$
where $g(L|K)$ is the number of distinct extensions of $v|K$ to $L$
and $Lv$ is a separable extension of $Kv$. Then $\Omega_{\cO_L|\cO_K}=0$.
\end{proposition}

\begin{proof} This follows from \cite[Proposition 5.6]{CK} and the fact    that the Galois group of the extension is the inertia group if and only if 
$[L:K]=g(L|K)[Lv:Kv]$ and $Lv$ is a separable extension of $Kv$ (as follows from \cite[Theorem (19.12)]{En}).
\end{proof}

The conclusion of \cite[Theorems 4.2, 4.3 and 1.4]{CKR}, with the assumption that $v$ is nondiscrete of rank 1 is:

\begin{theorem}\label{TheoremN1} Let $(L|K, v)$ be an Artin-Schreier defect extension or Kummer defect extension of degree $p$ such that $v$ is nondiscrete of rank 1,  with ramification
ideal $I_r$. Then there is an $\cO_L$-module isomorphism
$$
\Omega_{\cO_L|\cO_K}\cong I_r/I_r^p
$$
where $I_r$ is the ramification ideal of $\cO_L$, defined after Theorem \ref{theorem2*}.

Since $Lv$ has rank 1, there is an order preserving embedding of $Lv$ as a nondiscrete subgroup of the reals $\R$.
Then there exists (by the analysis of Sections 2.6 and 2.7 of \cite{CKR}) a nonnegative element $c\in \R$ such that 
\begin{equation}\label{eq1}
I_r=\{f\in \mathcal O_L\mid v(f)>c\}.
\end{equation}
We have that $\Omega_{\cO_L|\cO_K}=0$ if and only if $c=0$.
\end{theorem}

The conclusion of \cite[Theorem 4.5]{CK} in the case that $v$ is nondiscrete of rank 1 is:

\begin{theorem}\label{TheoremN2} Let $(L|K, v)$ be an Artin-Schreier extension of degree $p$ such that $v$ is nondiscrete of rank 1, with $f(L|K):=[Lv:Kv]=p$ and $Lv$ inseparable over $Kv$. Then 
$$
\Omega_{\cO_L|\cO_K}\cong \cO_L/(\vartheta^{1-p})
$$
where $\vartheta$ is a suitable  Artin-Schreier generator with $v(\vartheta)<0$. In particular, $\Omega_{\cO_L|\cO_K}\ne 0$.
\end{theorem}

The conclusion of \cite[Theorem 4.6]{CK} in the case that $v$ is nondiscrete of rank 1 is:

\begin{theorem}\label{TheoremN3} Let $(L|K, v)$ be an Artin-Schreier extension of degree $p$ such that $v$ is nondiscrete of rank 1, with 
$e(L|K):=(vL:vK)=p$. Then 
$$
\Omega_{\cO_L|\cO_K}\cong \vartheta^{-1}\mathcal M_L/(\vartheta^{-1}\mathcal M_L)^p
$$
where $\vartheta$ is a suitable  Artin-Schreier generator with $v(\vartheta)<0$ and $\mathcal M_L$ is the maximal ideal of $\cO_L$.
In particular, $\Omega_{\cO_L|\cO_K}\ne 0$.
\end{theorem}

The conclusion of \cite[Theorem 4.7]{CK} in the case that $v$ is nondiscrete is:

\begin{theorem}\label{TheoremN4} Let $(L|K, v)$ be a Kummer extension of prime degree $p$ such that $v$ is nondiscrete of rank 1, with 
$f(L|K):=[Lv:Kv]=p={\rm char }(Kv)$. Then there are two possible cases,
\begin{enumerate}
\item[i)]
$$
\Omega_{\cO_L|\cO_K}\cong \cO_L/(p)
$$
so $\Omega_{\cO_L|\cO_K}\ne 0$ or
\item[ii)]
$$
\Omega_{\cO_L|\cO_K}\cong \cO_L/(p\tilde c^{p-1})
$$
with $p\tilde c^{p-1}\in \cO_L$ and
$$
\Omega_{\cO_L|\cO_K}=0\mbox{ if and only if }(p\tilde c^{p-1})=\mathcal O_L\mbox{ if and only if }Lv\mbox{ is separable over }Kv.
$$
\end{enumerate}
\end{theorem}

The conclusion of \cite[Theorem 4.8]{CK} in the case that $v$ is nondiscrete is:

\begin{theorem}\label{TheoremN5} Let $(L|K, v)$ be a Kummer extension of prime degree $q$ such that $v$ is nondiscrete of rank 1, with 
$e(L|K):=(vL:vK)=q$. Then there are two possible cases
\begin{enumerate}
\item[i)] 
$$
\Omega_{\cO_L|\cO_K}\cong \mathcal M_L/q\mathcal M_L^q
$$
where $\mathcal M_L$ is the maximal ideal of $\cO_L$. We have that 
$$
\Omega_{\cO_L|\cO_K}=0\mbox{ if and only if } \mbox{char }(vK)\ne q.
$$
\item[ii)]
$$
\Omega_{\cO_L|\cO_K}\cong \xi^{-1}\mathcal M_L/q(\xi^{-1}\mathcal M_L)^q
$$
where $\mathcal M_L$ is the maximal ideal of $\cO_L$ and $\xi^{-1}\in \mathcal M_L$. 
In particular, 
$$
\Omega_{\cO_L|\cO_K}\ne 0.
$$
\end{enumerate}
\end{theorem}

\section{Almost Mathematics}\label{Section2}

Let $R$ be a rank 1 nondiscrete valuation ring with quotient field $K$ and maximal ideal $I$. We have that $(R,I)$ is a basic setup in the language of Chapter 2 of \cite{GR}. let $v$ be a valuation of $K$ such that $R$ is the valuation ring of $v$.

 Since $v$ is rank 1 nondiscrete, it has the property that if $\alpha\in I$ and $n\ge 1$, then there exists an element $\beta\in I$ such that $n v(\beta)<v(\alpha)$, so that $\frac{\alpha}{\beta^n}\in I$. This observation will be used repeatedly in this section. In particular, if $\epsilon\in I$, then we can factor $\epsilon=\alpha\beta$ where $\alpha,\beta\in I$.

An $R$-module $M$ is almost zero if $IM=0$.  Two $R$-modules $M$ and $N$ are said to be almost isomorphic if there exists an $R$-module  homomorphism $\phi:M\rightarrow N$ such that the kernel and cokernel of $\phi$ are almost zero.

If $M$ is an  $R$-module, then $M_*$ is defined to be $M_*={\rm Hom}_R(I,M)$. There is a natural homomorphism $i:M\rightarrow M_*$. The quotient $M_*/i(M)$ is almost zero; for $\psi\in M_*$ and $\epsilon\in I$, $\epsilon \psi=i(\psi(\epsilon))\in i(M)$. The kernel of $i:M\rightarrow M_*$ is $\{x\in M|I\subset\mbox{ ann}(x)\}$ which is almost zero, so $i:M\rightarrow M_*$ is an almost isomorphism.

\begin{lemma}\label{Lemma7} Suppose that $M$ is a torsion free $R$-module. Then there is a natural $R$-module isomorphism of $M_*$ with 
$\{m\in M\otimes_RK\mid \epsilon m\in M\mbox{ for all }\epsilon \in I\}$,
where $K$ is the quotient field of $R$.

If $M=A$ is an  $R$-algebra, then this identification gives $A_*$ a natural $R$-algebra structure extending that of $A$.
\end{lemma}

\begin{proof} Since $M$ is torsion free, there is a natural inclusion of $M$ into $M\otimes_RK$. Let 
$$
S=\{m\in M\otimes_RK\mid \epsilon m\in M\mbox{ for all }\epsilon \in I\}.
$$
Suppose that $\phi\in \mbox{Hom}_R(I,M)$. Then for nonzero $\delta,\epsilon \in I$, 
$\frac{\phi(\delta)}{\delta}=\frac{\phi(\epsilon)}{\epsilon}$. This determines an $R$-module homomorphism $\Psi:M_*\rightarrow M\otimes_RK$, given by $\Psi(\phi)=\frac{\phi(\epsilon)}{\epsilon}$ for $0\ne \epsilon\in I$. The map $\Psi$ is an injection since $M$ is torsion free.  For $0\ne \epsilon \in I$, $\epsilon\Psi(\phi)\in M$ so $\mbox{Image}(\Psi)\subset S$. Conversely, if $\lambda\in S$, we define $\phi\in \mbox{Hom}_R(I,M)$ by $\phi(\epsilon)=\epsilon\lambda$ for $\epsilon\in I$. By construction, $\Psi(\phi)=\lambda$. Thus $M_*\cong S$.
\end{proof}

Suppose that  $A$ is an $R$-algebra. There is a natural homomorphism of $A$ into $A_*$ defined by $x\mapsto \phi_x$ for $x\in A$, where $\phi_x(\epsilon)=\epsilon x$ for $\epsilon\in I$. The $R$-module  $A_*$ has an $R$-algebra structure, extending that of $A$, which is defined (Remark 2, page 187 \cite{Fa}) by
\begin{equation}\label{eq20}
(f\circ g)(\alpha\beta)=f(\alpha)g(\beta)\mbox{ for }f,g \in A_*\mbox{ and }\alpha,\beta\in I.
\end{equation}
This multiplication is the same as that  defined in Lemma \ref{Lemma7} if $A$ is a torsion free $R$-module.
If $A$ is an $R$-algebra and $M$ is an $A$-module, then $\mbox{Hom}_R(I,M)$ is an $A_*$-module. If $B$ is an $R$-algebra, then $\mbox{Hom}_R(I,B)$ is an $A_*$-algebra.

We give a simple example showing that we cannot assume that all naturally ocurring $R$-algebras are  $R$-torsion free. The computation of $A\otimes_RA$ is necessary to consider \'etale like properties of    $R\rightarrow A$ (Definition \ref{Def1}). 

\begin{example} Even if an $A$-algebra is an $R$-torsion free domain, it may be that $A\otimes_RA$ has $R$-torsion. 
\end{example} 

\begin{proof} Let $k$ be a field and $R$ be the polynomial ring $R=k[x,y]$. Let $A=R[w]/(xw-y)$. The ideal $(xw-y)$ is a prime ideal in $R[w]$, so $A\cong R[\frac{y}{x}]$. In particular, $A\cong k[x,\frac{y}{x}]$ is $R$-torsion free. However, 
$A\otimes_RA\cong A[z]/(xz-y)$ has $R$ torsion since $x(z-\frac{y}{x})=0$. 
\end{proof}

We now state some definitions which can be found in  Chapter 2 of  \cite{GR}.

\begin{definition}\label{Def1} Suppose that $A$ is an $R$-algebra and $M$ is an $A$-module. 
\begin{enumerate}
\item[1)] $M$ is almost flat if $\mbox{Tor}^A_{1}(M,N)$ is almost zero for all $A$-modules $N$. 
\item[2)] $M$ is almost projective if $\mbox{Ext}_A^1(M,N)$ is almost zero for all $A$-modules $N$.
\item[3)] An $A$-module $M$ is almost finitely generated (presented) if for each $\epsilon\in I$, there exists a finitely generated (presented) $A$-module $M_{\epsilon}$ and an $A$-module homomorphism $M_{\epsilon}\rightarrow M$ with kernel and cokernel annhilated by $\epsilon$.
\item[4)] An $A$-module $M$ is uniformly almost finitely generated if there exists a number $n$ such that for all $\epsilon\in I$, there exists a finitely generated $A$-module $M_{\epsilon}$ and an $A$-module homomorphism $M_{\epsilon}\rightarrow M$ with kernel and cokernel  annihilated by $\epsilon$ and  the number of generators of $M_{\epsilon}$ as an $A$-module is bounded above by $n$.
\end{enumerate}
\end{definition}

\begin{lemma}\label{Lemma6} Suppose that $A$ is an $R$-algebra and  $M$ is an $A$-module. Then the following are equivalent:
\begin{enumerate}
\item[1)] $M$ is an almost flat $A$-module.
\item[2)] For every homomorphism $U\rightarrow V$ of $A$-modules whose kernel is almost zero, the kernel of $U\otimes_AM\rightarrow V\otimes_AM$ is almost zero.
\item[3)] For every injective homomorphism $U\rightarrow V$ of $A$-modules, the kernel of $U\otimes_AM\rightarrow V\otimes_AM$ is almost zero.
\end{enumerate}
\end{lemma}

\begin{proof} First  suppose that $M$ is almost flat, so that $\mbox{Tor}_1^A(M,N)$ is almost zero for all $A$-modules $N$. Let $\gamma:U\rightarrow V$ be a homomorphism of $A$-modules whose kernel $K$ is almost zero. We have short exact sequences of $A$-modules
$$
0\rightarrow K\rightarrow U\rightarrow U/K\rightarrow 0\mbox{ and }
0\rightarrow U/K\rightarrow V\rightarrow V/\gamma(U)\rightarrow 0.
$$
Let $L$ be the image of ${\rm Tor}_1^A(V/\gamma(U),M)$ in $(U/K)\otimes_AM$ in the exact sequence
$$
{\rm Tor}_1^A(V/\gamma(U),M)\rightarrow (U/K)\otimes_AM\rightarrow V\otimes_AM.
$$
${\rm Tor}_1^A(V/\gamma(U),\,M)\cong {\rm Tor}_1^A(M,V/\gamma(U))$ is almost zero by our assumption on $M$, so $L$ is almost zero.
We have the following diagram of $A$-modules, where the column and row are exact:
$$
\begin{array}{ccccccc}
&&&&0&&\\
&&&&\downarrow&&\\
&&&&L&&\\
&&&&\downarrow&&\\
K\otimes_AM&\stackrel{\tau}{\rightarrow}&U\otimes_AM&\stackrel{\sigma}{\rightarrow}&(U/K)\otimes_AM&\rightarrow 0\\
&&&&\,\,\,\,\,\downarrow\phi&&\\
&&&&V\otimes_AM\\
\end{array}
$$
Suppose that $x\in\mbox{Kernel}(\gamma\otimes 1_M=\phi\sigma:U\otimes_AM\rightarrow V\otimes_AM)$ and $\epsilon\in I$. Write $\epsilon=\alpha\beta$ for some $\alpha,\beta\in I$. Then $\sigma(x)\in L$ implies $\alpha\sigma(x)=0$ and so $\sigma(\alpha x)=0$ which implies $\alpha x=\tau(z)$ for some $z\in K\otimes_AM$. We have that  $\beta z=0$ so $\epsilon x=\beta\alpha x=0$. 
So the kernel of $\gamma\otimes 1=\phi\sigma:U\otimes_AM\rightarrow V\otimes_AM$ is almost zero, and we have verified that condition 2) holds.

If 2) holds then certainly 3) holds. Suppose that condition 3) holds. Let $N$ be an $A$-module. We have a short exact sequence
$$
0\rightarrow K\stackrel{\alpha}{\rightarrow} P\rightarrow N\rightarrow 0
$$
where $P\rightarrow N$ is a surjection from a projective $A$-module $P$, giving an exact sequence
$$
{\rm Tor}_1^A(P,M)\rightarrow {\rm Tor}_1^A(N,M)\rightarrow K\otimes M
\stackrel{\alpha\otimes 1}\rightarrow  P\otimes M.
$$
The kernel of $\alpha\otimes 1$ is almost zero by the assumption on $M$  and $\mbox{Tor}_1^A(P,M)=0$ since $P$ is projective. Thus $\mbox{Tor}_1^A(M,N)\cong \mbox{Tor}_1^A(N,M)$ is almost zero.
\end{proof}

\begin{lemma}\label{Lemma13} Suppose that $A\rightarrow B$ and $B\rightarrow C$ are almost flat homomorphisms of $R$-algebras. Then $A\rightarrow C$ is almost flat.
\end{lemma}

\begin{proof}  Suppose that $M\rightarrow N$ is a homomorphism of $A$-modules whose kernel is almost zero. Then by Lemma \ref{Lemma6}, the kernel of $M\otimes_AB\rightarrow N\otimes_AB$ is almost zero and the kernel of $(M\otimes_AB)\otimes_BC\rightarrow (N\otimes_AB)\otimes_BC$ is almost zero. But this kernel is isomorphic to the kernel of $M\otimes_AC\rightarrow N\otimes_AC$.
\end{proof}

\begin{lemma}\label{Lemma14} Suppose that $\phi:A\rightarrow B$ is an almost flat homomorphism of $R$-algebras and $\psi:A\rightarrow C$ is a homomorphism of $R$-algebras. Then 
$$
1_C\otimes \phi:C\cong C\otimes_A A\rightarrow C\otimes_AB
$$
 is an almost flat homomorphism of $R$-algebras.
\end{lemma}

\begin{proof} We use the criterion of Lemma \ref{Lemma6}. Suppose that $M\rightarrow N$ is an injection of $C$-modules. Then the kernel $K$ of $M\otimes_AB\rightarrow N\otimes_AB$ is almost zero. But
$M\otimes_AB\cong M\otimes_C(C\otimes_AB)$ and $N\otimes_AB\cong N\otimes_C(C\otimes_AB)$ so $K$ is the kernel of $M\otimes_C(C\otimes_AB)\rightarrow N\otimes_C(C\otimes_AB)$. Thus
$1_C\otimes \phi:C\rightarrow C\otimes_AB$ is almost flat. 
\end{proof}


\begin{lemma}\label{Lemma8} Suppose that $f:A\rightarrow B$ and $g:B\rightarrow C$ are homomorphisms of $R$-algebras and that the multiplications $\mu_{B|A}:B\otimes_AB\rightarrow B$ and $\mu_{C|B}:C\otimes_BC\rightarrow C$ are almost flat. Then the multiplication $\mu_{C|A}:C\otimes_AC\rightarrow C$ is almost flat.
\end{lemma}

\begin{proof} The homomorphism $\mu_{C|A}$ factors as
$$
C\otimes_AC\stackrel{\phi}{\rightarrow}C\otimes_BC\stackrel{\mu_{C|B}}{\rightarrow}C
$$
where $\phi$ is the natural map. We will show that $\phi$ is almost flat. It will then follow that $\mu_{C|A}$ is almost flat by Lemma \ref{Lemma13} since it is a composition of almost flat homomorphisms.  

The natural commutative diagram
$$
\begin{array}{ccc}
C\otimes_BC&\stackrel{\phi}{\leftarrow}&C\otimes_AC\\
\uparrow&&\uparrow\\
B&\stackrel{\mu_{B|A}}{\leftarrow}&B\otimes_AB
\end{array}
$$
is Cartesian by  \cite[Proposition I.5.3.5]{EGA1}; that is, 
$$
C\otimes_BC\cong B\otimes_{B\otimes_AB}(C\otimes_AC).
$$
Thus $\phi$ is almost flat by Lemma \ref{Lemma14} since   $\mu_{B|A}$ is almost flat. 
\end{proof}

\begin{lemma}\label{Lemma12}(\cite[Lemma 2.4.15]{GR}) Let $A$ be an $R$-algebra and $M$ be an almost finitely generated $A$-module. Then $M$ is an almost projective $A$-module if and only if for every $\epsilon \in I$, there exists a positive integer $n(\epsilon)$ and $A$-module homomorphisms
\begin{equation}\label{eq22}
M\stackrel{u_{\epsilon}}{\rightarrow} A^{n(\epsilon)}\stackrel{v_{\epsilon}}{\rightarrow} M
\end{equation}
such that $v_{\epsilon}\circ u_{\epsilon}=\epsilon 1_M$.
\end{lemma}

\begin{proof} Suppose that the sequences (\ref{eq22}) exist for all $\epsilon\in I$. Let $N$ be an $A$-module. Then 
$\epsilon:\mbox{Ext}_A^1(M,N)\rightarrow \mbox{Ext}_A^1(M,N)$ factors as 
$$
{\rm Ext}_A^1(M,N)\stackrel{v_{\epsilon}^*}{\rightarrow}{\rm Ext}_A^1(A^{n(\epsilon)},N)
\stackrel{u_{\epsilon}^*}{\rightarrow} {\rm Ext}_A^1(M,N).
$$
Since $\mbox{Ext}_A^1(A^{n(\epsilon)},N)=0$, we have that $\epsilon \mbox{Ext}_A^1(M,N)=0$. Since this holds for all $\epsilon\in I$, $\mbox{Ext}_A^1(M,N)$ is almost zero for all $A$-modules $N$, so that $M$ is almost projective. 

Conversely, suppose that $M$ is almost finitely generated and almost projective. Suppose $\epsilon\in I$. Factor $\epsilon=\alpha\beta$ for some $\alpha,\beta\in I$. Since $M$ is almost finitely generated, there exists 
$n=n(\epsilon)\in \Z_{>0}$ and a homomorphism $\phi_{\alpha}:A^n\rightarrow M$ such that 
$\alpha\mbox{coker}(\phi_{\alpha})=0$. Let $M_{\alpha}=\phi_{\alpha}(A^n)$, giving a factorization of $\phi_{\alpha}$ as
$$
A^n\stackrel{\psi_{\alpha}}{\rightarrow} M_{\alpha}\stackrel{j_{\alpha}}{\rightarrow} M,
$$
 where $j_{\alpha}$ is the inclusion. For $x\in M$, $\alpha x=j_{\alpha}(y)$ for some $y\in M_{\alpha}$ since $\alpha\mbox{coker}(\phi_{\alpha})=0$, giving a factorization
 $$
 M\stackrel{\gamma_{\alpha}}{\rightarrow} M_{\alpha}\stackrel{j_{\alpha}}{\rightarrow}M
 $$
 of $\alpha 1_M:M\rightarrow M$. Let $K$ be the kernel of the surjection $\psi_{\alpha}$. We have an exact sequence
 $$
 {\rm Hom}_A(M,A^n)\stackrel{\psi_{\alpha}^*}{\rightarrow} {\rm Hom}_A(M,M_{\alpha})\rightarrow {\rm Ext}_A^1(M,K)
 $$
 and $\mbox{Ext}_A^1(M,K)$ is almost zero since $M$ is almost projective. Thus $\beta\gamma_{\alpha}$ is in the image of $\psi_{\alpha}^*$; that is, there exists an $A$-module homomorphism $u_{\epsilon}:M\rightarrow A^n$ such that $\psi_{\alpha}\circ u_{\epsilon}=\beta\gamma_{\alpha}$. Let $v_{\epsilon}=\phi_{\alpha}$. From the commutative diagram
 $$
 \begin{array}{ccccc}
 M&\stackrel{u_{\epsilon}}{\rightarrow}&A^n&\stackrel{\phi_{\alpha}}{\rightarrow}&M\\
 &\searrow\beta\gamma_{\alpha}&\downarrow \psi_{\alpha}& j_{\alpha}\nearrow\\
 &&M_{\alpha}
 \end{array}
 $$
 we see that $v_{\epsilon}\circ u_{\epsilon}=\alpha\beta 1_M=\epsilon 1_M$, giving the sequence (\ref{eq22}).
 \end{proof}

\begin{lemma}\label{Lemma5} Suppose that $A$ is an $R$-algebra and $M$ is an almost projective $A$-module. Then  $M$ is an almost flat $A$-module.
\end{lemma}

\begin{proof} The proof of the converse of Lemma \ref{Lemma12}, taking $\phi_{\alpha}:P\rightarrow M$ to be a surjection of a projective $A$-module $P$ onto $M$, shows that  for all $\epsilon\in I$,  there exists a factorization 
$$
M\stackrel{u_{\epsilon}}{\rightarrow} P\stackrel{\phi_{\alpha}}{\rightarrow} M
$$
such that $\phi_{\alpha}u_{\epsilon}=\epsilon 1_M$. Let $N$ be an $A$-module. Then there exists a factorization of 
$\epsilon:\mbox{Tor}_1^A(M,N)\rightarrow \mbox{Tor}_1^A(M,N)$ by 
$$
{\rm Tor}_1^A(M,N)\stackrel{u_{\epsilon}^*}{\rightarrow}{\rm Tor}_1^A(P,N)\stackrel{\phi_{\alpha}^*}{\rightarrow}
{\rm Tor}_1^A(M,N).
$$
Since a projective module is flat, $\mbox{Tor}_1^A(P,N)=0$. Thus $\epsilon\mbox{Tor}_1^A(M,N)=0$. Since this holds for all $\epsilon\in I$, $\mbox{Tor}_1^A(M,N)$ is almost zero for all $A$-modules $N$ and so $M$ is almost flat. 
\end{proof}

The following transparent proof of Lemma \ref{Lemma10} is by Hema Srinivasan.
\begin{lemma}\label{Lemma10}(\cite[Lemma 2.4.17]{GR})
Let $S$ be a commutative ring, $M$ an $S$-module and
$$
S^n\stackrel{\phi}{\rightarrow} S^m\rightarrow C\rightarrow 0
$$
be an exact sequence of $S$-modules. Let $C'$ be the cokernel of the dual homomorphism
$\phi^*:(S^m)^*\rightarrow (S^n)^*$. Then there is a natural isomorphism of $S$-modules
$$
{\rm Tor}_1^S(C',M)\cong {\rm Hom}_S(C,M)/{\rm Image}({\rm Hom}_S(C,S)\otimes M).
$$
\end{lemma}

\begin{proof} We have an exact sequence
$$
0\rightarrow {\rm Hom}_S(C,S)\stackrel{\Theta}{\rightarrow} (S^m)^*\stackrel{\phi^*}{\rightarrow} (S^n)^*\rightarrow C'\rightarrow 0.
$$

Let $P\rightarrow \mbox{Kernel}(\phi^*)$ be a surjection from a projective $S$-module $P$, giving   an exact sequence
$$
P\stackrel{\Lambda}{\rightarrow}(S^m)^*\stackrel{\phi^*}{\rightarrow} (S^n)^*\rightarrow C'\rightarrow 0.
$$
Then we have a commutative diagram
$$
\begin{array}{ccccccccc}
P&\stackrel{\Lambda}{\rightarrow}&\,\,\,\,\,\,\,\,(S^m)^*&\stackrel{\phi^*}{\rightarrow}& (S^n)^*&\rightarrow &C'&\rightarrow &0\\
&\searrow&\nearrow\Theta\\
&&{\rm Hom}_S(C,S)\\
\end{array}
$$
where the homomorphism $P\rightarrow {\rm Hom}_S(C,S)$ is a surjection and the homomorphism
 $\Theta:{\rm Hom}_S(C,S)\rightarrow (S^m)^*$ is an inclusion.
Tensoring this diagram with $M$ over $S$, we have a commutative diagram
$$
\begin{array}{ccccccccc}
P\otimes M&\stackrel{\Lambda\otimes 1}{\rightarrow}&\,\,\,\,\,\,\,\,(S^m)^*\otimes M&\stackrel{\phi^*\otimes 1}{\rightarrow}& (S^n)^*\otimes M&\rightarrow &C'\otimes M&\rightarrow &0\\
&\searrow&\nearrow \Theta\otimes 1\\
&&{\rm Hom}_S(C,S)\otimes M\\
\end{array}.
$$
We compute that 
$$
{\rm Tor}_1^S(C',M)={\rm Kernel}( \phi^*\otimes 1)/{\rm Image}(\Lambda\otimes 1)=
{\rm Kernel}(\phi^*\otimes 1)/{\rm Image}(\Theta\otimes 1).
$$
Let $K={\rm Kernel}(\phi^*\otimes 1)$.  We  define $S$-module homomorphisms
$$
\alpha=\alpha_m:{\rm Hom}_S(S^m,S)\otimes M\rightarrow {\rm Hom}_S(S^m,M)
$$
by 
$$
\alpha(\gamma\otimes x)(u)=\gamma(u)x
$$
for $\gamma\in {\rm Hom}(S^m,S)$, $x\in M$ and  $u\in S^m$. 
The $\alpha_m$  are  isomorphisms of $S$-modules by Proposition 2 (ii), Chapter II, Section 4.2, page 269 \cite{Bou}. 
The diagram of $S$-modules
$$
\begin{array}{ccccccc}
0&\rightarrow&{\rm Hom}_S(C,M)&\rightarrow &{\rm Hom}_S(S^m,M)&\stackrel{\phi^*}{\rightarrow}&{\rm Hom}_S(S^n,M)\\&&&&\uparrow \alpha_m&&\uparrow \alpha_n\\
0&\rightarrow&K&\rightarrow&(S^m)^*\otimes M&\stackrel{\phi^*\otimes 1}{\rightarrow}&(S^n)^*\otimes M
\end{array}
$$
commutes. Since the rows are exact, we then have an isomorphism $K\cong \mbox{Hom}_S(C,M)$, giving the conclusions of the lemma.
\end{proof}

\begin{proposition}\label{Prop5}(\cite[Proposition 2.4.18]{GR}) Let $A$ be an $R$-algebra. Then
every almost finitely presented almost flat $A$-module is almost  projective.
\end{proposition}

\begin{proof} Let $M$ be an almost finitely presented almost flat $A$-module.
Suppose that $\epsilon\in I$. Factor $\epsilon=\alpha^2\beta$ with $\alpha, \beta\in I$, There exists a finitely presented $A$-module $M_{\alpha}$  and exact sequences of $A$-modules
$$
0\rightarrow K_{\alpha}\rightarrow M_{\alpha}\stackrel{f_{\alpha}}{\rightarrow} M\rightarrow C_{\alpha}\rightarrow 0
$$
where $\alpha K_{\alpha}=\alpha C_{\alpha}=0$ and 
$$
A^m\stackrel{\Psi}{\rightarrow} A^n\stackrel{\Lambda}{\rightarrow} M_{\alpha}\rightarrow 0
$$
for some positive integers $m$ and $n$. Thus the natural isomorphism $f_{\alpha}(M_{\alpha})\cong M_{\alpha}/K_{\alpha}$ induces  natural inclusions 
$$
\alpha M\subset M_{\alpha}/K_{\alpha}\subset M.
$$ 
Let $C'$ be the cokernel of the dual homomorphism $\Psi^*:(A^n)^*\rightarrow (A^m)^*$. Lemma \ref{Lemma10} implies that
$$
{\rm Hom}_A(M_{\alpha},M)/{\rm Im}({\rm Hom}_A(M_{\alpha},A)\otimes_AM)\cong {\rm Tor}_1^A(C',M).
$$
This last module is almost zero since $M$ is almost flat. Thus $\beta f_{\alpha}$ is the image of an element $\sum_{j=1}^n\phi_j\otimes m_j\in {\rm Hom}_A(M_{\alpha},A)\otimes_AM$. Define $v:M_{\alpha}\rightarrow A^n$ by $v(x)=(\phi_1(x),\ldots,\phi_n(x))$ for $x\in  M_{\alpha}$ and $w:A^n\rightarrow M$ by $w(y_1,\ldots,y_n)=\sum_{j=1}^ny_jm_j$. Let $\lambda:M\rightarrow A^n$ be the composition
$$
M\stackrel{\alpha}{\rightarrow} M_{\alpha}/K_{\alpha}\stackrel{\alpha}{\rightarrow}M_{\alpha}\stackrel{v}{\rightarrow} A^n.
$$
Then $w\circ\lambda=\alpha^2\beta 1_M=\epsilon1_M$. By Lemma \ref{Lemma12}, $M$ is an almost projective $A$-module.
\end{proof}

\begin{lemma}\label{Lemma4} Let $A\rightarrow B$ be a homomorphism of $R$-algebras such that the multiplication $\mu_{B|A}:B\otimes_AB\rightarrow B$ makes $B$ an almost flat $B\otimes_AB$-module. Then $\Omega_{B|A}$ is almost zero.
\end{lemma}

\begin{proof} Let $J$ be the kernel of  $\mu_{B|A}:B\otimes_AB\rightarrow B$ and let $S=B\otimes_AB$  so that we have a short exact sequence
$$
0\rightarrow J\rightarrow S\rightarrow S/J\rightarrow 0
$$ 
of $S$-modules where $S/J$ is an almost flat $S$-module.  Tensoring with $S/J$ over $S$, we have an exact sequence
$$
\mbox{Tor}^S_1(S/J,S/J)\rightarrow J/J^2\rightarrow S/J\rightarrow S/J\rightarrow 0
$$
where $\mbox{Tor}_1^S(S/J,S/J)$ is almost zero since $S/J$ is an almost flat $S$-module. We have that $J/J^2\rightarrow S/J$ is the zero map. Thus $\Omega_{B|A} =J/J^2$ is  almost zero. 
\end{proof}

The following definition is  in Chapter 3 of \cite{GR}.
\begin{definition}\label{Def2} An extension $A\rightarrow B$ of $R$-algebras is almost finite \'etale if 
\begin{enumerate}
\item[1)] (Almost finite projectiveness) $B$ is an almost finitely presented almost projective $A$-module.
\item[2)] (Almost unramifiedness) $B$ is an almost projective $B\otimes_AB$-module by the multiplication map $\mu_{B|A}:B\otimes_AB\rightarrow B$.
\end{enumerate}
\end{definition}

A homomorphism $A\rightarrow B$ satisfies that $B$ is a finitely presented and projective $A$-module and $\mu_{B|A}:B\otimes_AB\rightarrow B$ makes $B$ a projective $B$-module if and only if $A\rightarrow B$ is finite \'etale (\cite[Proposition IV.18.3.1]{EGA4}). This is also proven in \cite{F}.

\begin{proposition}\label{Prop6}(\cite[Proposition 3.1.4]{GR})  A homomorphism $A\rightarrow B$ of $R$-algebras is almost unramified ($B$ is an almost projective $B\otimes_AB$-module under the multiplication $\mu_{B|A}:B\otimes_AB\rightarrow B$) if and only if there exists  an element $e\in (B\otimes_AB)_*$ such that $e^2=e$, $\mu_*(e)=1$ and $e\mbox{Kernel}(\mu)_*=0$ where $\mu=\mu_{B|A}:B\otimes_AB\rightarrow B$ is the multiplication map and $\mu_*:(B\otimes_AB)_*\rightarrow B_*$ is the induced homomorphism.
\end{proposition}

\begin{proof} 


We have a short exact sequence of $B\otimes_AB$-modules
$$
0\rightarrow J_{B/A}\rightarrow B\otimes_AB\stackrel{\mu_{B|A}}{\rightarrow} B\rightarrow 0.
$$

First suppose that $\phi:A\rightarrow B$ is unramified; that is, $B$ is almost projective under the homomorphism $\mu_{B|A}:B\otimes_AB\rightarrow B$. Let $\epsilon\in I$. The proof of the converse of Lemma \ref{Lemma12}, taking $\phi_{\alpha}=\mu_{B|A}:B\otimes_AB\rightarrow B$, shows that  there exists a factorization of homomorphisms of $B\otimes_AB$-modules, 
$$
B\stackrel{u_{\epsilon}}{\rightarrow}B\otimes_AB\stackrel{\mu_{B|A}}{\rightarrow} B
$$
of $\epsilon 1_B$.

 Let $e_{\epsilon}=u_{\epsilon}(1)$. We  have that $\mu_{B|A}(e_{\epsilon})=\epsilon 1$.
 
 Since $B$ is a $B\otimes_AB$-module under the action $xy=\mu_{B|A}(x)y$ for $x\in B\otimes_AB$ and $y\in B$, and $u_{\epsilon}$ is $B\otimes_AB$-linear, $u_{\epsilon}(xy)=x u_{\epsilon}(y)$. Thus
$$
e_{\epsilon}^2=e_{\epsilon}u_{\epsilon}(1)=u_{\epsilon}(\mu_{B|A}(e_{\epsilon})1)=\mu_{\epsilon}(\mu_{B|A}(e_{\epsilon})=\epsilon e_{\epsilon}.
$$
For $x\in J_{B/A}$, we have 
$$
xe_{\epsilon}=xu_{\epsilon}(1)=u_{\epsilon}(\mu_{B|A}(x)1)=0.
$$
Thus $e_{\epsilon}J_{B/A}=0$. 

Suppose $\delta,\epsilon\in I$ with corresponding elements $e_{\delta},e_{\epsilon}\in B\otimes_AB$. Then 
$\delta 1-e_{\delta}$, $\epsilon 1-e_{\epsilon}\in J_{B/A}$ which implies 
\begin{equation}\label{eq21}
\delta e_{\epsilon}=\epsilon e_{\delta}
\end{equation}
for all $\delta,\epsilon\in I$. 

Define $e\in (B\otimes_AB)_*=\mbox{Hom}_R(I,B\otimes_AB)$ by 
$$
e(\lambda)=\alpha e_{\beta}
$$
for $\lambda\in I$, if $\lambda=\alpha\beta$ with $\alpha,\beta\in I$. We will verify that this map is well defined. 
Suppose that $\lambda=\alpha\beta=\delta\epsilon$ for some $\lambda,\alpha,\beta,\delta,\epsilon \in I$.
We have that $\alpha$ divides  $\delta$ or $\delta$ divides $\alpha$ in $R$. Without loss of generality, we may assume that $\alpha$ divides $\delta$.
We have that $ae_{\epsilon}=\epsilon e_a$ by (\ref{eq21}). Thus
$$
\delta e_{\epsilon}=\frac{\delta}{\alpha}\alpha e_{\epsilon}=\frac{\delta}{\alpha}\epsilon e_{\alpha}=\beta e_{\alpha}=\alpha e_{\beta}.
$$
 We now verify that $e$ is an $R$-module homomorphism. Suppose that $x\in R$ and $\lambda\in I$. Factor $\lambda=\alpha\beta$ with $\alpha,\beta\in I$. We have 
$$
e(x\lambda)=(x\alpha)e_{\beta}=x(\alpha e_{\beta})=xe(\lambda).
$$
Suppose that $\lambda_1,\lambda_2\in I$.
We can find $\alpha,\beta_1,\beta_2\in I$ such that
$\lambda_1=\alpha\beta_1$ and $\lambda_2=\alpha\beta_2$. Then
$$
e(\lambda_1+\lambda_2)=(\beta_1+\beta_2)e_{\alpha}=\beta_1e_{\alpha}+\beta_2e_{\alpha}=e(\lambda_1)+e(\lambda_2).
$$
Thus $e$ is an $R$-module homomorphism, and so $e\in (B\otimes_AB)_*$.
For $\lambda\in I$, let $\lambda=\alpha\beta$ with $\alpha,\beta\in I$. Write $\alpha=\alpha_1\gamma$ and $\beta=\beta_1\gamma$ with $\alpha_1,\beta_1,\gamma\in I$. Then by (\ref{eq20}),
$$
e^2(\lambda)=(e\circ e)(\alpha\beta)=e(\alpha)e(\beta)=(\alpha_1e_{\gamma})(\beta_1e_{\gamma})
=\alpha_1\beta_1e_{\gamma}^2=\alpha_1\beta_1\gamma e_{\gamma}=e(\lambda).
$$
Thus $e^2=e$. For $\lambda\in I$,
$$
[(\mu_{B|A})_*(e)](\lambda)=\mu_{B|A}(e(\lambda))=\mu_{B|A}(\alpha e_{\beta})=\alpha \mu_{B|A}(e_{\beta})=\alpha\beta=\lambda.
$$
Thus $(\mu_{B|A})_*(e)=1$. For $x\in (J_{B/A})_*$ and $\lambda\in I$, write $\lambda=\alpha\beta\gamma$. Then by (\ref{eq20}),
$$
(xe)(\lambda)= x(\alpha)e(\beta\gamma)=x(\alpha)\beta e_{\gamma}=0
$$
since $x(\alpha)\beta\in J_{B / A}$. Thus $e(J_{B/A})_*=0$. We have shown that $e$ satisfies the three conditions of the proposition.


Conversely, suppose that $e\in (B\otimes_AB)_*$ has the given properties. For $\epsilon\in I$, let 
$$
e_{\epsilon}=\epsilon e=e(\epsilon)\in B\otimes_AB.
$$
We have that  $\mu_{B|A}(e_{\epsilon})=\epsilon$ and $e_{\epsilon}\mbox{Kernel}(\mu_{B|A})=0$.  Define $u_{\epsilon}:B\rightarrow B\otimes_AB$ by $u_{\epsilon}(b)=e_{\epsilon}(1\otimes b)$ for $b\in B$ and $v_{\epsilon}:B\otimes_AB\rightarrow B$ by $v_{\epsilon}=\mu_{B|A}$. For $x\in B\otimes_AB$ and $b\in B$,
$$
u_{\epsilon}(xb)=u_{\epsilon}(\mu_{B|A}(x)b)=\epsilon e (1\otimes \mu_{B|A}(x)b)=\epsilon e(x(1\otimes b))
=x\epsilon e(1\otimes b)
$$
since $x-1\otimes \mu_{B|A}(x)\in J_{B/A}$. Thus $u_{\epsilon}$ is a $B\otimes_AB$-module homomorphism.
For $b\in B$, $v_{\epsilon}\circ u_{\epsilon}(b)=\epsilon b$. 
Thus $B$ is an almost projective $B\otimes_AB$-module by Lemma \ref{Lemma12}.

\end{proof}

The following definition is in Chapter 3 of \cite{GR}.

\begin{definition}\label{Def3} A homomorphism  $A\rightarrow B$ of $R$-algebras is weakly \'etale if $B$ is an almost flat $A$-module and $B$ is an almost flat $B\otimes_AB$-module by the multiplication map $\mu_{B|A}$.
\end{definition}

\section{Extensions of valuation rings}\label{Section3}

\begin{lemma}\label{Lemma3}(Chapter VI, Section 3, no. 6, Lemma 1 \cite{BouC}) Let $A$ be a valuation ring. All torsion free finitely generated $A$-modules are free. All finitely generated ideals of $A$ are principal. All torsion free $A$-modules are flat.
\end{lemma}

\begin{lemma}\label{Lemma50} Let $A\rightarrow B$ be an extension of valuation rings. Then $A\rightarrow B$ is faithfully flat. 
\end{lemma}

\begin{proof} This follows from Lemma \ref{Lemma3} and  \cite[Theorem 7.2]{Mat}.
\end{proof}

\begin{lemma}\label{Lemma1} Suppose that $A\rightarrow B$ is an extension of domains, with respective quotient fields $K$ and $L$ such that $B$ is a flat $A$-module. Then the natural homomorphism $B\otimes_AB\rightarrow L\otimes_KL$ is an injection.
\end{lemma}

\begin{proof}

Tensor the injection $B\rightarrow L$ with $\otimes_AB$ to get an injection $B\otimes_AB\rightarrow L\otimes_AB$. 

The field $L$ is a flat $A$-module, since $K$ is a flat $A$-module and $L$ is a flat $K$-module, so $L\otimes_AB$ is a flat $B$-module. Now tensor the injection $B\rightarrow L$ with $(L\otimes_AB)\otimes_B$ to get an injection $L\otimes_AB\rightarrow L\otimes_AL$.

We have that $L\otimes_AL\cong (L\otimes_AK)\otimes_KL$ and $L\otimes_AK$ is just the localization $S^{-1}L$ where $S=K\setminus\{0\}$, so $L\otimes_AK\cong L$. Thus $L\otimes_AL\cong L\otimes_KL$ and so we have an injection $B\otimes_AB\rightarrow L\otimes_KL$.
\end{proof}

Let $L/K$ be a finite  field extension. The trace form $t_{L/K}:L\times L\rightarrow K$ is defined by $t_{L/K}(\alpha,\beta)=\mbox{Trace}_{L/K}(\alpha\beta)$ for $\alpha,\beta\in L$. Here $\mbox{Trace}_{L/K}$ is the trace of $L$ over $K$. The form $t_{L/K}$ is a nondegenerate symmetric bilinear form if $L/K$ is separable (Proposition 2.8 of Chapter 1 \cite{N}). Suppose that $L/K$ is a finite separable field extension. Let $e_1,\ldots,e_n$ be a basis of $L$ as a vector space over $K$. Let $A$ be the matrix of $t_{L/K}$ with respect to the basis $\{e_1,\ldots,e_n\}$. We have that $\mbox{det}(A)\ne 0$ since $t_{L/K}$
is nondegenerate (Chapter XIII, Section 6, Proposition 6.1\cite{L}). Let $B=(b_{ij})$ be the inverse matrix of $A$. Let $\{e_1^*,\ldots,e_n^*\}$ be the basis defined by $e_i^*=\sum_{k=1}^nb_{ki}e_k$. Then
$$
t_{L/K}(e_ie_j^*)=\delta_{ij}
$$
where $\delta_{ij}$ is the Kronecker delta; that is, $\{e_j^*\}$ is the dual basis to $\{e_i\}$ for the form $t_{L/K}$.

\begin{proposition}\label{Prop4}(\cite[Proposition 6.3.8]{GR})
 Suppose that $(L/K,v)$ is a finite separable  extension of valued fields such that $v$ has rank 1 and is not discrete. Let the basic setup be $(\cO_K,I)$ where $I=\cM_K$.
Let $W_L$ be the integral closure of $\cO_K$ in $L$. Then $W_L$ is an almost finitely presented and uniformly almost finitely generated $\cO_K$-module which admits the uniform bound $n=[L:K]$. Further, $W_L$ is an almost projective $\cO_K$-module. 
\end{proposition}

\begin{proof}  Let $t_{L/K}$ be the trace form of $L/K$ defined above the statement of this proposition. Let $\{e_1,\ldots,e_n\}$ be a basis of the $K$-vector space $L$ such that $e_1,\ldots,e_n\in W_L$. Let $\{e_1^*,\ldots,e_n^*\}$ be the dual basis defined above the statement of this proposition. There exists $a\in \cO_K$ such that $ae_i^*\in W_L$ for all $i$. Let $w\in W_L$. We can write $w=\sum_{i=1}^na_ie_i$ for some $a_i\in K$. Since $wae_j^*\in W_L$, we have that 
$t_{L/K}(wa,e_j^*)=\mbox{Trace}_{L/K}(wae_j^*)\in \cO_K$ for all $j$ (by Theorem 4, Section 3, Chapter IV, page 260 \cite{ZS1}). We also have that $t_{L/K}(wa,e_j^*)=aa_j$. Thus
\begin{equation}\label{eq10}
e_1\cO_K+\cdots+e_n\cO_K\subset W_E\subset a^{-1}(e_1\cO_K+\cdots+e_n\cO_K).
\end{equation}
We can write $W_E$ as the direct limit of the family $\mathcal W$ of all the finitely generated $\cO_K$-submodules of $W_E$ containing $e_1,\ldots,e_n$. If $W_0\in \mathcal W$, then $W_0$ is a free $\cO_K$-module  by Lemma \ref{Lemma3}.
By (\ref{eq10}), $W_0\otimes_{\cO_K}K\cong K^n$. Since $W_0$ is a free $\cO_K$-module of finite rank, we have that $W_0$ is a free $\cO_K$-module of rank $n$.

To show that $W_L$ is an almost finitely presented and uniformly almost finitely generated $\cO_K$-module which admits the uniform bound $n=[L:K]$,  it suffices to prove the following assertion. 
\begin{equation}\label{eq11}
\mbox{Let $\epsilon \in I$. Then there exists $W'\in \mathcal W$ such that $\epsilon W_L\subset W'$.}
\end{equation}
Suppose that the condition of (\ref{eq11}) does not hold for some $\epsilon$. Set $W_0=\sum_{i=1}^ne_i\cO_K$. Since (\ref{eq11}) does not hold, there exists $g_1\in W_L$ such that $\epsilon g_1\not\in W_0$. Let $W_1=W_0+g_1\cO_K$. Again since (\ref{eq11}) does not hold, there exists $g_2\in W_L$ such that $\epsilon g_2\not\in W_1$. Let $W_2=W_1+g_2\cO_K$. Continuing this way, we construct a sequence 
$$
W_0=\sum_{i=1}^ne_i\cO_K\subset W_1\subset \cdots \subset W_m\subset\cdots
$$
indexed by the non negative integers
such that $W_i\in \mathcal W$ for all $i$ and $\epsilon W_{i+1}\not\subset W_i$ for all $i$. By (\ref{eq10}), we have short exact sequences of $\cO_K$-modules
$$
0\rightarrow aW_{k+1}/aW_0\rightarrow (\cO_K)^n/a(\cO_K)^n\rightarrow (\cO_K)^n/aW_{k+1}\rightarrow 0.
$$
The fitting ideals $F_i(M)$ of a module $M$ are defined in  \cite[Appendix D]{K}. By  \cite[Proposition D17]{K},
$$
\begin{array}{l}
F_0((\cO_K)^n/a(\cO_K)^n)= F_0(aW_{k+1}/aW_0)F_0((\cO_K)^n/aW_{k+1})\\\subset F_0(aW_{k+1}/aW_0)=F_0(W_{k+1}/W_0).
\end{array}
$$
Induction on $i$ in the short exact sequences
$$
0\rightarrow W_i/W_0\rightarrow W_{i+1}/W_0\rightarrow W_{i+1}/W_i\rightarrow 0
$$
and \cite[Proposition D17]{K} shows that 
$$
F_0(W_{k+1}/W_0)= \prod_{i=0}^kF_0(W_{i+1}/W_i)
$$
for all $k\ge 0$. We have that
$$
\mbox{Ann}_{\mathcal O_K}(W_{i+1}/W_i)=\{\alpha\in \mathcal O_K\mid \alpha g_{i+1}\in W_i\}\subset \epsilon\mathcal O_K
$$
since $\mbox{Ann}_{\mathcal O_K}(W_{i+1}/W_i)$ is an ideal in the valuation ring $\mathcal O_K$ and $\epsilon g_{i+1}\not\in W_i$.

By  \cite[Proposition D14]{K}, 
$$
F_0(W_{i+1}/W_i)\subset \mbox{Ann}_{\cO_K}(W_{i+1}/W_i)\subset \epsilon \cO_K
$$
for all $i\ge 0$. Thus
$$
a^n\cO_K=F_0((\cO_K)^n/a(\cO_K)^n)\subset \prod_{i=0}^{k-1}F_0(W_{i+1}/W_i)\subset \epsilon^k\cO_K
$$
for all $k\ge 0$, which implies that $nv(a)\ge kv(\epsilon)$ for all $k\ge 0$ which is impossible since $v$ has rank 1. Thus statement (\ref{eq11}) is true. Since each $W'\in \mathcal W$ is a free $\cO_K$-module of rank $n$, $W_L$ is an almost finitely presented and uniformly almost finitely generated $\cO_K$-module. $W_L$ is a flat $\cO_K$-module by Lemma \ref{Lemma3}. Thus $W_L$ is an almost  projective $\cO_K$-module by Proposition   \ref{Prop5}.
\end{proof}

An extension of valued fields $(L/K,v)$ is unibranched if $v|K$ has a unique extension to $L$.

\begin{proposition}\label{Prop3} Suppose that $(L/K,v)$ is a finite separable unibranched extension of valued fields and that $v$ has rank 1 and is nondiscrete.  Let the basic setup be $(\cO_K,I)$ where $I=\cM_K$.
If $\Omega_{\cO_L|\cO_K}$ is almost zero, then $\cO_K\rightarrow \cO_L$ is almost finite \'etale. 
\end{proposition}

\begin{proof} Let $A=\cO_K$. The integral closure $W_L$ of $A$ in $L$ is $\mathcal O_L$ since $L/K$ is unibranched. 
By Proposition \ref{Prop4}, $B=\cO_L$ is a uniformly almost finitely generated $A$-module which admits the uniform bound $n=[L:K]$. 

Given $\epsilon\in I$, let $W'$  be the finitely generated $A$-submodule of $B$ such that $\epsilon B\subset W'$ of (\ref{eq11}) of the proof of Proposition \ref{Prop4}. Let $B_{\epsilon}$ be the $A$-algebra generated by $W'$. $B_{\epsilon}$ is generated as an $A$-algebra by finitely many elements $g_1,\ldots,g_n\in B$, which are thus integral over $A$. Thus $B_{\epsilon}$ is a finitely generated $A$-module, so that $B_{\epsilon}$ is a free $A$-module by Lemma \ref{Lemma3}.  Since $B_{\epsilon}\otimes_{A}K\cong K^n$, we have that $B_{\epsilon}\cong A^n$ as an $A$-module. Further, $\epsilon B\subset B_{\epsilon}$. 
We have that $B\otimes_AB$ is a torsion free $A$-module since $A\rightarrow B$ is flat. We further have natural inclusions
$$
B\otimes_AB\subset (B\otimes_AB)_*\subset L\otimes_KL
$$
by Lemmas \ref{Lemma1} and \ref{Lemma7}.

We have a commutative diagram
$$
\begin{array}{ccccccccc}
0&\rightarrow&J_{\epsilon}&\rightarrow &B_{\epsilon}\otimes_{A}B_{\epsilon}&\stackrel{u_{\epsilon}}{\rightarrow}&
B_{\epsilon}&\rightarrow & 0\\
&&\downarrow&&\downarrow&&\downarrow&&\\
0&\rightarrow&J&\rightarrow &B\otimes_{A}B&\stackrel{u}{\rightarrow}&
B&\rightarrow & 0\\
\end{array}
$$
where $u_{\epsilon}$ and $u$ are the multiplication homomorphisms, the rows are short exact and the vertical arrows are injections (since both $B_{\epsilon}\otimes_A B_{\epsilon}\rightarrow L\otimes_KL$ and $B\otimes_AB\rightarrow L\otimes_KL$ are injections by Lemma \ref{Lemma1}).

The fact that $\epsilon B\subset B_{\epsilon}$ implies $\epsilon^2(B\otimes_AB)\subset B_{\epsilon}\otimes_AB_{\epsilon}$ and $\epsilon^2J\subset B_{\epsilon}\otimes_AB_{\epsilon}$. The fact that $u(J)=0$ implies $\epsilon^2J\subset J_{\epsilon}$ which implies $\epsilon^4 J^2\subset J_{\epsilon}^2$.  $\epsilon\Omega_{B|A}=0$ where $\Omega_{B|A}= J/J^2$ implies $\epsilon J\subset J^2$. Thus
$\epsilon^5 J_{\epsilon}\subset \epsilon^5J\subset \epsilon^4J^2\subset J_{\epsilon}^2$.

We have that $J_{\epsilon}$ is a finitely generated ideal, which is generated by at most $n$ elements by  \cite[Lemma 8.1.4]{F}, since $B_{\epsilon}$ is generated by $\le n$ elements as an $A$-algebra. 
So there exist $f_1,\ldots,f_n\in B_{\epsilon}\otimes_AB_{\epsilon}$ such that $J_{\epsilon}=(f_1,\ldots,f_n)$. Let $\lambda=\epsilon^5\otimes 1\in B_{\epsilon}\otimes_AB_{\epsilon}$. 
Since $\lambda J_{\epsilon}\subset J_{\epsilon}^2$, we have that 
for $1\le i\le n$,  $\lambda f_i=\sum_{j=1}^n a_{ij}f_j$ with $a_{ij}\in J_{\epsilon}$. Let $b_{ij}=\delta_{ij}\lambda-a_{ij}$ where $\delta_{ij}$ is the Kronecker delta. Then $\mbox{Det}(b_{ij})\in B_{\epsilon}\otimes_AB_{\epsilon}$ and $\mbox{Det}(b_{ij})=\lambda^n+c_1\lambda^{n-1}+\cdots+c_n$ where $c_i\in J_{\epsilon}$.
For $1\le i\le n$, we have that $\sum_{j=1}^n(\delta_{ij}\lambda-a_{ij})f_j=0$. Multiplying the matrix $(b_{ij})$ on the left by its adjoint, we obtain that $\mbox{Det}(b_{ij})f_k=0$ for $1\le k\le n$ and so $\mbox{Det}(b_{ij})$ annihilates $J_{\epsilon}$.

 Let $e_{\epsilon}=\mbox{Det}(b_{ij})$. We have that 
$$
u_{\epsilon}(e_{\epsilon})=\epsilon^{5n}.
$$
Further,
 $$
 e_{\epsilon}^2= e_{\epsilon}(\lambda^n+c_1\lambda^{n-1}+\cdots+c_n)=\lambda^ne_{\epsilon}=\epsilon^{5n}e_{\epsilon}.
 $$
 For $f\in J_*$ we have that $\epsilon f\in J$ and thus $\epsilon^3f\in J_{\epsilon}$ so that $\epsilon^3e_{\epsilon}f=0$ which implies $e_{\epsilon}f=0$ since $L\otimes_KL$ is $I$-torsion free, as $L$ is $I$-torsion free and $L$ is a flat $K$-module. Thus 
 $$
 e_{\epsilon}J_*=0.
 $$
 Now for $\delta$ and $\epsilon$ in $I$, $\delta^{5n}\otimes 1 - e_{\delta}$ and $\epsilon^{5n}\otimes 1 -e_{\epsilon}\in J$ 
 which implies that 
 $$
 (\delta^{5n}\otimes 1-e_{\delta})e_{\epsilon}=0=
 (\epsilon^{5n}\otimes 1-e_{\epsilon})e_{\delta}
 $$
 and so 
 $$
 \delta^{5n}e_{\epsilon}=\epsilon^{5n}e_{\delta}
 $$
 for all $\delta,\epsilon\in I$. Let $e \in L\otimes_KL$ be the element such that 
 $$
 e=\frac{1}{\epsilon^{5n}}e_{\epsilon}
 $$
 for all $0\ne \epsilon\in I$. By our calculations for $e_{\epsilon}$,
 $$
 e^2=\left(\frac{1}{\epsilon^{5n}}e_{\epsilon}\right)^2
 =\left(\frac{1}{\epsilon^{5n}}\right)^2e_{\epsilon}^2=\left(\frac{1}{\epsilon^{5n}}\right)^2\epsilon^{5n}e_{\epsilon}=\frac{1}{\epsilon^{5n}}e_{\epsilon}=e,
 $$
 $$
 u(e)=u\left(\frac{1}{\epsilon^{5n}}e_{\epsilon}\right)=\frac{1}{\epsilon^{5n}}u(e_{\epsilon})=1,
 $$
 and 
 $$
eJ_*= \frac{1}{\epsilon^{5n}}e_{\epsilon}J_*=0.
 $$
 Suppose $\alpha\in I$. Then there exist $\epsilon\in I$ such that $v(\epsilon^{5n})<v(\alpha)$. Thus 
 $$
 \alpha e=\frac{\alpha}{\epsilon^5}e_{\epsilon}\in B_{\epsilon}\otimes_AB_{\epsilon}\subset B\otimes_AB
 $$
 which implies that $e\in (B\otimes_AB)_*$ by Lemma \ref{Lemma7}. Thus $\cO_K\rightarrow \cO_L$ is almost unramified by Proposition \ref{Prop6}. By our construction of the $B_{\epsilon}$ and Proposition \ref{Prop5}, $\cO_K\rightarrow \cO_L$ is almost finite projective. Thus $\cO_K\rightarrow\cO_L$ is almost finite \'etale.
\end{proof}

We restate Theorem \ref{Theorem3'} of the introduction here for the reader's convenience. 

\begin{theorem}\label{Theorem3} Suppose that $(L/K,v)$ is a finite separable unibranched extension of valued fields and that $v$ has rank 1 and is nondiscrete.  Let the basic setup be $(\cO_K,I)$ where $I=\cM_K$. Then $\cO_K\rightarrow \cO_L$ is almost finite \'etale if and only if $\Omega_{\cO_L|\cO_K}$ is almost zero.
\end{theorem}

\begin{proof} This follows from Lemmas \ref{Lemma4} and \ref{Lemma5} and by 
Proposition \ref{Prop3}.
\end{proof}

\begin{lemma}\label{Lemma9} Suppose that $K$ is a valued field and $K^h$ is its henselization. Then 
the multiplication map $\cO_{K^h}\otimes_{\cO_K}\cO_{K^h}\rightarrow \cO_{K^h}$ is flat.
\end{lemma}

\begin{proof} By Theorem 1, page 87 \cite{Ra}, there exist \'etale extensions $\cO_K\rightarrow A_i$ and maximal ideals $m_i$ of $A_i$ such that the henselization of $\cO_K$ is $(\cO_K)^h=\displaystyle\lim_{\rightarrow} (A_i)_{m_i}$. By Equation (\ref{eq40}), $\cO_{K^h}=(\cO_K)^h$. Let $A=\displaystyle\lim_{\rightarrow} A_i$, $m=\displaystyle\lim_{\rightarrow}m_i$ so that $(\cO_K)^h=A_m$. Now each extension $\cO_K\rightarrow A_i$ being \'etale means that each $A_i$ is a finitely presented $\cO_K$-algebra, $A_i$ is a flat $\cO_K$-module and $\Omega_{A_i|\cO_K}=0$ (\cite[Corollary 17.6.2 and Theorem 17.4.2]{EGA4}). Thus the multiplication $A_i\otimes_{\cO_K}A_i\rightarrow A_i$ makes $A_i$ a projective $A_i\otimes_{\cO_K}A_i$-module (\cite[Proposition 4.1.2 and Theorem 8.3.6]{F}). We then conclude that $A_i$ is a flat $A_i\otimes_{\cO_K}A_i$-module (for instance by the analog of Lemma \ref{Lemma5} in ordinary mathematics, e.g.  \cite[Corollary 3.46]{R}).
 By Proposition 9, Chapter I, Section 2, no 7
\cite{BouC}, we have that 
$A=\displaystyle\lim_{\rightarrow} A_i$ is a flat $\displaystyle\lim_{\rightarrow} (A_i\otimes_{\cO_K}A_i)$-module. Now
$$
\lim_{\rightarrow}(A_i\otimes_{\cO_K}A_i)\cong (\lim_{\rightarrow} A_i)\otimes_{\cO_K}(\lim_{\rightarrow} A_i)
$$
by Proposition 7, Chapter II, Section 6.3, page 204   \cite{Bou}. 
Thus $A$ is a flat $A\otimes_{\cO_K}A$-module. 

We have that $A_m\otimes_AA_m\cong A_m$ (since $\otimes_AA_m$ is just localization by $A\setminus m$) so that the multiplication $A_m\otimes_AA_m\rightarrow A_m$ is an isomorphism, so in particular is flat.  By the analog of  Lemma \ref{Lemma8} in ordinary mathematics ((ii) of Lemma 1.2 \cite{F2}) applied to the composition 
$\cO_K\rightarrow A\rightarrow A_m$ with the basic setup $I=\cO_K$, we conclude that $A_m\otimes_{\cO_K}A_m\rightarrow A_m$ is flat; that is,  $\cO_{K^h}\otimes_{\cO_K}\cO_{K^h}\rightarrow \cO_{K^h}$ is flat.
\end{proof}

We will require the following Proposition in the proof of Theorem \ref{Theorem4'}

\begin{proposition}\label{PropN2} Suppose that $(K,v)$ is a valued field where $v$ is nondiscrete of rank 1. Let $K\rightarrow L$ be an Artin-Schreier or Kummer extension. Identify $v$ with an extension of $v$ to $L$. Then $\Omega_{\mathcal O_L|\mathcal O_K}$ is zero if $\Omega_{\mathcal O_L|\mathcal O_K}$ is almost zero.
\end{proposition}

\begin{proof} By the proof of Theorem \ref{Theorem1*}, it suffices to  show that $\Omega_{\cO_L|\cO_K}$  is zero if $\Omega_{\mathcal O_L|\mathcal O_K}$ is almost zero in the cases of Proposition \ref{PropN1} and Theorems \ref{TheoremN1} - \ref{TheoremN5} of Section \ref{SectionN}.

The most difficult case is that of Theorem \ref{TheoremN1}. Suppose that we are in that situation.
We have that $\Omega_{\cO_L|\cO_K}=I_r/I_r^p$ where $I_r=\{f\in \cO_L\mid v(f)>c\}$ for some $c\ge 0$.
Suppose $c>0$. Then there exist $x\in I_r$ such that $v(x)<\frac{3}{2}c$ and $\epsilon\in I$ such that $v(\epsilon)<(p-2+\frac{1}{2})c$. Thus 
$$
\epsilon x\not\in \{f\in \cO_L\mid v(f)>pc\}=I_r^p,
$$
and so $I_r/I_r^p$ is not almost zero. Thus if $\Omega_{\cO_L|\cO_K}$ is almost zero, then $c=0$ so $\Omega_{\cO_L|\cO_K}=I/I^p$. Suppose $f\in I$. Then $f=\alpha\beta^p$ for some $\alpha,\beta\in I$, so $f\in I^p$. Thus $\Omega_{\cO_L|\cO_K}=0$.

The remaining cases are proven in a similar way, using the explicit structure of $\Omega_{\cO_L|\cO_K}$ given in these theorems. 

\end{proof}

The following  theorem gives a simpler proof of   \cite[Proposition 6.6.2]{GR}, stated as Theorem \ref{GRThm1} in the introduction.  Our proof of 2) implies 3) is of a different nature than the proof in \cite{GR}. They use their theory of the different ideal $\mathcal D_{B/A}$ of an almost finite projective morphism to prove this direction.

\begin{theorem}\label{Theorem4'}(\cite[Proposition 6.6.2]{GR}) Suppose that $(K,v)$ is a valued field, where $v$ is nondiscrete of rank 1. Let the basic setup be $(\cO_K,I)$ where $I=\cM_K$.  Identify $v$ with an extension of $v$ to a separable closure $K^{\rm sep}$ of $K$. Then the following are equivalent.
\begin{enumerate}
\item[1)] $\Omega_{\cO_{K^{\rm sep}}|\cO_K}$ is zero.
\item[2)] $\Omega_{\cO_{K^{\rm sep}}|\cO_K}$ is almost zero.
\item[3)] $\cO_{K}\rightarrow \cO_{K^{\rm sep}}$ is weakly \'etale.
\end{enumerate}
\end{theorem}

\begin{proof} 

First suppose that $\mathcal O_K\rightarrow \mathcal O_{K^{\rm sep}}$ is weakly \'etale. By Lemma \ref{Lemma4}, $\Omega_{\cO_{K^{\rm sep}}|\cO_K}$ is almost zero.

Now suppose that $\Omega_{\cO_{K^{\rm sep}}|\cO_K}$ is almost zero.

Let $K^h$ be the henselization of $K$. Then $\Omega_{\cO_{K^{\rm sep}}|\cO_K}\cong\Omega_{\cO_{K^{\rm sep}}|\cO_K^h}$ by  \cite[Proposition 5.10]{CK}.  By Lemma \ref{Lemma9}, 
the multiplication $\cO_{K^h}\otimes_{\cO_K}\cO_{K^h}\rightarrow \cO_{K^h}$ is flat.

  By Lemma \ref{Lemma8}, it suffices to prove that the multiplication $\cO_{K^{\rm sep}}\otimes_{\cO_{K^h}}\cO_{K^{\rm sep}}\rightarrow \cO_{K^{\rm sep}}$
  is almost flat to conclude that $\mathcal O_K\rightarrow \mathcal O_{K^{\rm sep}}$ is weakly \'etale. Thus we may assume that $K$ is henselian.

We have that
$\cO_{K^{\rm sep}}=\displaystyle\lim_{\rightarrow}\cO_L$ where the limit is over the set $S$ of finite Galois subextensions $L/K$ of $K^{\rm sep}$. By Equation (\ref{eq41}), 
\begin{equation}\label{eq27}
\Omega_{\cO_{K^{\rm sep}}|\cO_K}\cong \lim_{\rightarrow}\left(\Omega_{\cO_L|\cO_K}\otimes_{\cO_L}\cO_{K^{\rm sep}}\right).
\end{equation}
where the limit over $L\in S$.

 Suppose that $L/K$ is a finite Galois extension with $L\subset K^{\rm sep}$. We will show that
\begin{equation}\label{eq26}
\Omega_{\cO_L|\cO_K}\otimes_{\cO_L}\cO_{K^{\rm sep}} \rightarrow \Omega_{\cO_{K^{\rm sep}}|\cO_K}
\end{equation}
is an injection. Suppose not. Then there exists a Galois extension $M/K$ such that $L$ is a subextension and
$$
\Omega_{\mathcal O_L|\mathcal O_K}\otimes_{\mathcal O_K}\mathcal O_{K^{\rm sep}}
\rightarrow
\Omega_{\mathcal O_M|\mathcal O_K}\otimes_{\mathcal O_K}\mathcal O_{K^{\rm sep}}
$$ 
is not an injection. Since extensions of valuation rings are faithfully flat, 
$$
\Omega_{\mathcal O_L|\mathcal O_K}\otimes_{\mathcal O_K}\mathcal O_M\rightarrow \Omega_{\mathcal O_M|\mathcal O_K}
$$
 is not an injection. But this is a contradiction to Theorem \ref{Theorem4*}.

 Since $\cO_{L}$ is a faithfully flat $\cO_K$ module for all $L\in S$,  we have that $\Omega_{\cO_L|\cO_K}$ is almost zero for all $L\in S$. By Proposition \ref{Prop3} or Theorem \ref{Theorem3}, $\cO_K\rightarrow \cO_L$ is almost \'etale for all $L\in S$ (as $K \rightarrow L$ is unibranched since $K$ is henselian). Thus $\cO_L$ is an almost projective $\cO_L\otimes_{\cO_K}\cO_L$-module, so that 
$\cO_L$ is an almost flat $\cO_L\otimes_{\cO_K}\cO_L$-module by Lemma \ref{Lemma5}. We will now prove that $\cO_K\rightarrow \cO_{K^{\rm sep}}$ is weakly \'etale.
Let
$$
0\rightarrow U\rightarrow V
$$
be an injection of $\cO_{K^{\rm sep}}\otimes_{\cO_K}\cO_{K^{\rm sep}}$-modules. We must show that the kernel $\rm{Ker}$ of the $\cO_{K^{\rm sep}}$-module homomorphism 
$$
U\otimes_{(\cO_{K^{\rm sep}}\otimes_{\cO_K}\cO_{K^{\rm sep}})}\cO_{K^{\rm sep}}
\rightarrow
V\otimes_{(\cO_{K^{\rm sep}}\otimes_{\cO_K}\cO_{K^{\rm sep}})}\cO_{K^{\rm sep}}
$$
is almost zero. For $L\in S$, we may regard $U$ and $V$ as $\cO_L\otimes_{\cO_K}\cO_L$-modules. Let $\rm{Ker}_L$ be the kernel of the $\cO_L$-module homomorphism
$$
U\otimes_{(\cO_{L}\otimes_{\cO_K}\cO_{L})}\cO_{L}
\rightarrow
V\otimes_{(\cO_{L}\otimes_{\cO_K}\cO_{L})}\cO_{L},
$$
so for $L,M\in S$ with $L\subset M$, we have commutative diagams
$$
\begin{array}{ccccccc}
0&\rightarrow &{\rm Ker}_L&\rightarrow & U\otimes_{(\cO_{L}\otimes_{\cO_K}\cO_{L})}\cO_{L} &\rightarrow&
V\otimes_{(\cO_{L}\otimes_{\cO_K}\cO_{L})}\cO_{L}\\
&&\downarrow&&\downarrow&&\downarrow\\
0&\rightarrow &{\rm Ker}_M&\rightarrow & U\otimes_{(\cO_{M}\otimes_{\cO_K}\cO_{M})}\cO_{M} &\rightarrow&
V\otimes_{(\cO_{M}\otimes_{\cO_K}\cO_{M})}\cO_{M}
\end{array}
$$
where the rows are exact. By  \cite[Theorem A2]{Mat}, we have an exact sequence
\begin{equation}\label{eq12}
0\rightarrow \lim_{\rightarrow}\rm{Ker}_L\rightarrow \lim_{\rightarrow}\left(U\otimes_{(\cO_{L}\otimes_{\cO_K}\cO_{L})}\cO_{L}\right)\rightarrow \lim_{\rightarrow}
\left(V\otimes_{(\cO_{L}\otimes_{\cO_K}\cO_{L})}\cO_{L}\right)
\end{equation}

where the limits are over $L\in S$. Now $\displaystyle\lim_{\rightarrow}\cO_L=\cO_{K^{\rm sep}}$ and 
$$
\lim_{\rightarrow}\left(\cO_L\otimes\cO_L\right)\cong \left(\lim_{\rightarrow}\cO_L\right)\otimes_{\cO_K}\left(\lim_{\rightarrow}\cO_L\right)=\cO_{K^{\rm sep}}\otimes_{\cO_K}\cO_{K^{\rm sep}}
$$
by Proposition 7 of Chapter II, Section 6 no. 3 \cite{Bou}.  Applying this proposition again to (\ref{eq12}), we have the exact sequence
$$
0\rightarrow \lim_{\rightarrow} \rm{Ker}_L\rightarrow U\otimes_{\displaystyle\lim_{\rightarrow}\left(\cO_L\otimes_{\cO_K}\cO_L\right)}\left(\lim_{\rightarrow}\cO_L\right)\rightarrow
V\otimes_{\displaystyle\lim_{\rightarrow}\left(\cO_L\otimes_{\cO_K}\cO_L\right)}\left(\lim_{\rightarrow}\cO_L\right)
$$
giving the exact sequence
$$
0\rightarrow \lim_{\rightarrow}\rm{Ker}_L\rightarrow 
U\otimes_{(\cO_{K^{\rm sep}}\otimes_{\cO_K}\cO_{K^{\rm sep}})}\cO_{K^{\rm sep}}
\rightarrow
V\otimes_{(\cO_{K^{\rm sep}}\otimes_{\cO_K}\cO_{K^{\rm sep}})}\cO_{K^{\rm sep}}
$$
so that $\rm{Ker}=\lim_{\rightarrow}\rm{Ker}_L$. Since $\rm{Ker}_L$ are all almost zero, 
$\rm{Ker}$ is almost zero. Thus $\cO_{K^{\rm sep}}$ is an almost flat $\cO_{K^{\rm sep}}\otimes_{\cO_K}\cO_{K^{\rm sep}}$-module (by Lemma \ref{Lemma6}), and so $\cO_K\rightarrow \cO_{K^{\rm sep}}$ is weakly \'etale.

We will now show that if $\Omega_{\cO_{K^{\rm sep}}|\cO_K}$ is almost zero, then $\Omega_{\cO_{K^{\rm sep}}|\cO_K}$ is zero, from which the equivalence of 1) and 2) follows. We assume that $\Omega_{\cO_{K^{\rm sep}}|\cO_K}$
is almost zero but not  zero, and will derive a contradiction. We will repeatedly make use of the fact that if $K\rightarrow L\rightarrow M$ is a tower of valued fields, then $\cO_L\rightarrow \cO_M$ is faithfully flat, so that $\Omega_{\cO_L|\cO_K}$ is almost zero if and only if $\Omega_{\cO_L|\cO_K}\otimes_{\cO_L}\cO_M$ is almost zero  and $\Omega_{\cO_L|\cO_K}$ is  zero if and only if $\Omega_{\cO_L|\cO_K}\otimes_{\cO_L}\cO_M$ is  zero.

By Equations (\ref{eq26}) and (\ref{eq27}), if $L\in S$,
 there exists an injection $\Omega_{\mathcal O_L|\cO_K}\otimes_{\cO_L}\cO_{K^{\rm sep}}\rightarrow
 \Omega_{\cO_{K^{\rm sep}}|\cO_K}$. Thus $\Omega_{\cO_{K^{\rm sep}}|\cO_K}$
 almost zero but not  zero implies that there exists a finite Galois extension $L$ of $K$ such that $\Omega_{\cO_L|\cO_K}$ is almost zero but not zero. After possibly extending $K\rightarrow L$ to a larger finite Galois  extension $K\rightarrow M$, so that $M$ contains enough primitive roots of unity, there is a tower of field extensions
 $$
 K\rightarrow M_0\rightarrow M_1\rightarrow \cdots\rightarrow M_n=M
 $$
 where $M_0$ is the inertia field of $M/K$ and each extension $M_i\rightarrow M_{i+1}$ is an Artin-Schreier extension or a Kummer extension of prime degree (Proposition \ref{Prop3*}). We have that $\Omega_{\cO_L|\cO_K}\otimes_{\cO_K}\cO_M$ is a submodule of   $\Omega_{\cO_M|\cO_K}$  and $\Omega_{\cO_M|\cO_K}\otimes_{\cO_M}\cO_{K^{\rm sep}}$ is a submodule of  $\Omega_{\cO_{K^{\rm sep}}|\cO_K}$
 by Theorem \ref{Theorem4*}, so $\Omega_{\cO_M|\cO_K}$ is almost zero but not zero.  We have that $\Omega_{\cO_{M_0}|\cO_K}=0$ by Equation (\ref{eq42}), so 
 $\Omega_{\cO_M|\cO_K}=\Omega_{\cO_M|\cO_{M_0}}$ by  \cite[Theorem 25.1]{Mat}. By Theorem \ref{Theorem4*}, we thus have that 
 $\Omega_{\cO_{M_{i+1}}|\cO_{M_i}}$ is almost zero for all $i$ but some  $\Omega_{\cO_{M_{i+1}}|\cO_{M_i}}$ is not zero.   
 By Proposition \ref{PropN2}, this is not possible.
 Thus $\Omega_{\cO_{K^{\rm sep}}|\cO_K}$ is zero. 
 \end{proof}

 If $(K,v)$ is a valued field where $v$ is discrete of rank 1, and $\mathfrak m$ is the maximal ideal of $\mathcal O_K$, then $(\mathcal O_K,\mathfrak m)$ is not a basic setup (as defined in Chapter 2 of \cite{GR}). In this case, we must use the basic setup $(\cO_K,\cO_K)$, so that we are just doing ordinary mathematics. In particular, being ``almost zero'' is the same as being zero.
 
 The following lemma shows that the equivalent conditions of Theorem \ref{Theorem4'} can never occur if $(K,v)$ is discrete of rank 1.
 
 \begin{lemma} Suppose that $(K,v)$ is a valued field where $v$ is discrete of rank 1. Then $\Omega_{O_{K^{\rm sep}}|\mathcal O_K}$ is not zero and $\mathcal O_K\rightarrow \mathcal O_{K^{\rm sep}}$ is not weakly \'etale.
 \end{lemma}

 \begin{proof}
 We will show that either condition $\Omega_{O_{K^{\rm sep}}|\mathcal O_K}$ is not zero or $\mathcal O_K\rightarrow \mathcal O_{K^{\rm sep}}$ is not weakly \'etale
  implies that $v$ is not discrete, giving a contradiction. Suppose that $v$ is discrete (of rank 1).  Let $q$ be a prime which is relatively prime to the residue degree of $\cO_K$ and let $K\rightarrow L$ be a finite Galois extension such that $L$ contains all $q$-th roots of unity. The value group $vL$ is discrete  since $vK$ is discrete.  Let $u\in L$ be such that $v(u)$ is a generator of $vL$, and let $L\rightarrow M$ be the degree $q$ Kummer extension $M=L(\sqrt[q]{u})$. We have  that the ramification index of $M$ over $L$ is $e(M|L)=q$. Thus we are in case 1 of  \cite[Theorem 4.7]{CK} and by this theorem and our construction we have that $\Omega_{\cO_M|\cO_L}$ is not zero. There is a natural surjection of $\cO_M$-modules $\Omega_{\cO_M|\cO_K}\rightarrow \Omega_{\cO_M|\cO_L}$, by the first fundamental exact sequence,  \cite[Theorem 25.1]{Mat}. By Theorem \ref{Theorem4*}, $\Omega_{\cO_M|\cO_K}\otimes_{\cO_M}\cO_{K^{\rm sep}}\rightarrow \Omega_{\cO_{K^{\rm sep}}|\cO_K}$ is an injection. Thus $\Omega_{\cO_{K^{\rm sep}}|\cO_K}\ne 0$.
 
 
 The proof of Lemma \ref{Lemma4} extends to show that if $A\rightarrow B$ is a homomorphism of rings and $B$ is a flat $B\otimes_AB$-module, then $\Omega_{B/A}=0$. The condition that $\mathcal O_K\rightarrow \mathcal O_{K^{\rm sep}}$ is weakly \'etale is just that $\mathcal O_{K^{\rm sep}}$ is a flat $\mathcal O_{K^{\rm sep}}\otimes_{\mathcal O_K}\mathcal O_{K^{\rm sep}}$-module (since we have  the basic set up $(\mathcal O_K,I=\mathcal O_K))$. Thus $\mathcal O_K\rightarrow \mathcal O_{K^{\rm sep}}$ is not weakly \'etale.
 \end{proof}

\section{Deeply ramified extensions of a local field}\label{Section4}

In this section, we consider the equivalent conditions characterizing deeply ramified fields, as they define deeply ramified fields  in \cite{CG}, and show that they are indeed the algebraic extensions $K$ of the p-adics $\Q_p$  which satisfy $\Omega_{\cO_{\overline{\Q_p}}|\cO_K}=0$,
for an algebraic closure $\overline{\Q_p}$ of $\Q_p$ which contains $K$. Since $\Q_p$ is henselian, there is a unique extension of the $p$-adic valuation of $\Q_p$ to a valuation $v$ of an algebraic closure $\overline{\Q_p}$ of $\Q_p$.

We begin by recalling the construction used in  \cite{CG} to analyze an algebraic extension $K$ of $\Q_p$. We express
$K=\cup_{n=0}^{\infty}F_n$ as a union of finite algebraic field extensions $F_n$ of $\Q_p$ such that $F_n\subset F_{n+1}$ for all $n$.
Then each $\cO_{F_n}$ is the integral closure of $\Z_p$ in $F_n$, and is a discrete valuation ring. We have that $\cup \cO_{F_n}=\cO_K$.

Now let $K'$ be a finite extension of $K$. 
In \cite[Chapter V, Section 4, Lemma 6]{Se}, it is shown that the extension $K'/K$ can be realized as follows. 
Let $r=[K':K]$ and $\omega_1,\ldots,\omega_r$ be a basis of $K'$ over $K$. We have relations
$$
\omega_i\omega_j=\sum c_{ij}^{k}\omega_k
$$
with $c_{ij}^k\in K$. Choose $n_0$ large enough that all of the $c_{ij}^k$ are in $F_{n_0}$. Then define $F_{n_0}'$ to be the $r$-dimensional $F_{n_0}$ vector subspace $F_{n_0}'=F_{n_0}\omega_1+\cdots +F_{n_0}\omega_r$ of $K'$. By our construction, $F_{n_0}'$ is a subfield of $K'$.  
 Then define $F_{n}'=F_{n_0}'F_n$ for $n\ge n_0$. Thus $F_{n}'=F_{n}\omega_1+\cdots +F_{n}\omega_r$. Since $\omega_1,\ldots,\omega_r$ are linearly independent over $K$,  
 we  have that
\begin{equation}\label{eqCG1}
F_n'\otimes_{F_n}K\cong F_n'K=K'
\end{equation}
for $n\ge n_0$.

Let $\delta_n=\delta(F_n'/F_n)$ be the different of $\cO_{F_n'}$ over $\cO_{F_n}$. The different is defined in \cite[Chapter III, Section 3]{Se} or in \cite[Chapter V, Section 11]{ZS1}. Since $\cO_{F_n}$ and $\cO_{F_n'}$ are  discrete valuation rings, we have that
\begin{equation}\label{eqCG3}
\Omega_{\cO_{F_n'}|\cO_{F_n}}\cong \cO_{F_n'}/(\delta(F_n'/F_n))
\end{equation}
by \cite[Chapter III, Section 7, Proposition 14]{Se}. By \cite[Lemma 2.6]{CG} and its proof, $v(\delta_n)$ is a decreasing sequence for $n\ge n_0$ and   $\lim_{n\rightarrow \infty}v(\delta_n)$ does not depend on our choices of $\{F_n\}$ or $\{F_n'\}$.

We may now state the theorem defining deeply ramified fields in \cite{CG}. They call the  fields satisfying the equivalent conditions of this theorem as deeply ramified fields, which is the origin of this name. We will call the fields satisfying the equivalent conditions of Theorem \ref{CGtheorem}, Coates Greenberg deeply ramified fields. 

\begin{theorem}(page 143, \cite{CG})\label{CGtheorem} Let $\overline\Q_p$ be an algebraic closure of $\Q_p$ and let $v$ be the unique extension of the $p$-adic valuation to $\overline \Q_p$. Suppose that $K$ is an algebraic extension of $\Q_p$ which is contained in $\overline \Q_p$. Then the following are equivalent.
\begin{enumerate}
\item[i)] $K$ does not have finite conductor.
\item[ii)] $v(\delta(F_n/\Q_p))\rightarrow 0$ as $n\rightarrow 0$.
\item[iii)] $H^1(K,\cM_K)=0$
\item[iv)] for every finite extension $K'$ of $K$ we have ${\rm Tr}_{K'/K}(\cM_{K'})=\cM_K$.
\item[v)] for every finite extension $K'$ of $K$ we have $v(\delta(F_n'/F_n))\rightarrow 0$ as $n\rightarrow 0$.
\end{enumerate}
\end{theorem}

Let $G$ be the Galois group of $\overline \Q_p$ over $\Q_p$. For $w\in [-1,\infty)$, let $G^{(w)}$ be the $w$-th ramification group of $G$ in the upper numbering (\cite[Chapter 4, Section 3]{Se}) and $\Q_p^{(w)}$ be the fixed field of $G^{(w)}$. The field $K$ is said to have finite conductor if $K\subset \Q_p^{(w)}$ for some $w\in [-1,\infty)$.

We will show that the equivalent conditions of Theorem \ref{CGtheorem} are indeed equivalent to 
$$
\Omega_{\cO_{\overline{\Q_p}}|\cO_K}=0.
$$
Specifically, we will show that $\Omega_{\cO_{\overline{\Q_p}}|\cO_K}=0$ if and only if condition v) of Theorem \ref{CGtheorem} holds. 

\begin{theorem}\label{localfieldDR} Suppose that $K$ is an algebraic extension of $\Z_p$. Then $K$ is Coates Greenberg deeply ramified if and only if $K$ is deeply ramified; that is, if and only if $\Omega_{\cO_{\overline{\Q_p}}|\cO_K}=0$.
\end{theorem}

Before proving this theorem, we will study in more detail a finite algebraic extension $K'$ of an algebraic extension $K$ of $\Z_p$.
Let $\{F_n\}$ and $\{F_n'\}$ be the fields constructed before the statement of Theorem \ref{CGtheorem}.

Let $r=[K':K]$. We have by (\ref{eqCG1}) that $F_n'\otimes_{F_n}K\cong F_n'K=K'$ for all $n\ge n_0$. Let $S_n=\mathcal O_{F_n'}\otimes_{\mathcal O_{F_n}}\mathcal O_K$ for $n\ge n_0$. Since $\mathcal O_{F_n'}$ is a flat $\mathcal O_{F_n}$-module, for all $n\ge n_0$, there exists a basis $w(n)_1,\ldots,w(n)_r$ of $\mathcal O_{F_n'}$ as an $\mathcal O_{F_n}$-module. Thus
$S_n$ is a free $\mathcal O_K$-module with basis $w(n)_1,\ldots,w(n)_r$. By Lemmas \ref{Lemma50} and \ref{Lemma1}, 
we have  inclusions
$$
S_n=\mathcal O_{F_n'}\otimes_{\mathcal O_{F_n}}\mathcal O_K\subset F_n'\otimes_{\mathcal O_{F_n}}K=K'.
$$
We have that $\cup S_n=\cO_{K'}$  since $\cO_{K'}$ is the integral closure of  $\cO_{K}$ in  $K'$. Since $\cO_F\rightarrow \cO_K$ is flat, we have by \cite[Proposition 16.4]{E} and (\ref{eqCG3}),  that 
$$
\Omega_{S_n|\mathcal O_K}\cong \left(\Omega_{\mathcal O_{F_n}|\mathcal O_F}\right)\otimes_{\mathcal O_F}\mathcal O_K
\cong\mathcal O_{F_n}\otimes_{\mathcal O_F}\mathcal O_K/\delta_n\mathcal O_{F_n}\otimes_{\mathcal O_F}\mathcal O_K
\cong S_n/\delta_n S_n.
$$

\begin{proposition}\label{CGC} Let $K'$ be a finite extension of $K$ and assume that $vK$ is not discrete. Then $v(\delta_n)\rightarrow 0$ as $n\rightarrow \infty$ if and only if $\Omega_{\mathcal O_{K'}|\mathcal O_K}$ is almost zero.
\end{proposition}

\begin{proof}
Since $\cup S_i=\mathcal O_{K'}$, the proof of Proposition \ref{Prop4} with the family $\mathcal W$ replaced by the family $\Sigma=\{S_i\mid i\ge n_0\}$, shows that (\ref{eq12}) of Proposition \ref{Prop4} holds for the family $\Sigma$; that is, for every $\epsilon\in \cM_K$, there exists $S_i\in\Sigma$ such that $\epsilon\mathcal O_{K'}\subset S_i$.
So 
\begin{equation}\label{eqCG2}\mbox{Given $\epsilon\in\cM_K$, there exists $\sigma(\epsilon)\in \Z_{>0}$ such that $i\ge \sigma(\epsilon)$ implies $\epsilon\mathcal O_{K'}\subset S_i$.}
\end{equation}
Lemmas \ref{Lemma50} and \ref{Lemma1} imply that for all $n$, $S_n\otimes_{\cO_K}S_n\subset \cO_{K'}\otimes_{\cO_K}\cO_{K'}$. Thus we have a commutative diagram where the rows are injections and the columns are short exact,
$$
\begin{array}{ccc}
0&&0\\
\downarrow&&\downarrow\\
J_n&\rightarrow & J\\
\downarrow&&\downarrow\\
S_n\otimes_{\cO_K}S_n&\rightarrow&\cO_{K'}\otimes_{\cO_K}\cO_{K'}\\
\downarrow&&\downarrow \\
S_n&\rightarrow &S\\
\downarrow&&\downarrow\\
0&&0
\end{array}
$$
where the homomorphisms $S_n\otimes_{\cO_K}S_n\rightarrow S_n$ and $\cO_{K'}\otimes_{\cO_K}\cO_{K'}\rightarrow \cO_{K'}$ are the multiplication maps. Thus $\Omega_{S_n|\cO_K}\cong J_n/J_n^2$ and $\Omega_{\cO_{K'}|\cO_K}\cong J/J^2$.

For $\epsilon\in \cM_K$, $\epsilon^2(\cO_{K'}\otimes_{\cO_K}\cO_{K'})\subset S_n\otimes_{\cO_K}S_n$ if $n\ge \sigma(\epsilon)$ 
which implies that $\epsilon^2J\subset J_n$ if $n\ge \sigma(\epsilon)$.

Suppose that $\Omega_{\cO_{K'}|\cO_K}$ is almost zero. 
Let $\epsilon\in \cM_{K}$. Then $\epsilon\Omega_{\cO_{K'}|\cO_K}=0$, which   implies 
$\epsilon J \subset J^2$ which implies
$$
\epsilon^4J_n\subset \epsilon^4J\subset \epsilon^4J^2\subset J_n^2
$$
which implies that  $v(\delta_n)\le 4v(\epsilon)$ for $n\ge \sigma(\epsilon)$ and so $v(\delta_n)\rightarrow 0$ as $n\rightarrow \infty$.

Now suppose that $v(\delta_n)\rightarrow 0$ as $n\rightarrow\infty$. Given $\lambda,\epsilon\in \cM_{K'}$, there exists $l\in \Z_{>0}$ such that $n\ge l$ implies $v(\delta_n)<\lambda$ and $\epsilon^2J\subset J_n$. Thus 
$$
  \delta_n\epsilon^2J\subset \delta_nJ_n\subset J_n^2\subset J^2
  $$
  which implies that $\delta_n\epsilon^2\Omega_{\cO_{K'}|\cO_K}=0$. Thus $\lambda\epsilon^2\Omega_{\cO_L|\cO_K}=0$, and so
  $\Omega_{\cO_{K'}|\cO_K}$ is almost zero.
\end{proof}

\begin{lemma}\label{LemmaCGA} The field $K$ is Coates Greenberg deeply ramified if and only if for every finite Galois extension $K'$ of $K$, $v(\delta(F_n'/F_n))\rightarrow 0$ as $n\rightarrow \infty$.
\end{lemma}

\begin{proof} Suppose that  $v(\delta(F_n'/F_n))\rightarrow 0$ as $n\rightarrow \infty$ for every finite Galois extnsion $K'$ of $K$ and $K''$ is a finite extension of $K$. Let $K\rightarrow K'$ be a finite Galois extension of $K$ such that $K''$ is a subextension. Using the construction before the statement of Theorem \ref{CGtheorem}, we can find finite field extensions $\{F_n\}$ of $\Z_p$ such that $\cup F_n=K$, and find finite field extensions 
$F_n\rightarrow F_n''\rightarrow F_n'$ for $n\ge  n_0$ for some $n_0$ such that $\cup F_n''=K''$, $\cup F_n'=K'$, and the other properties of the families given before the statement of Theorem \ref{CGtheorem} hold for the family  $\{F_n''\}$ for $K''$ and for the family $\{F_n'\}$ for $K'$. By the transitivity formula for Dedekind domains of \cite[Theorem 31, page 309]{ZS1} or \cite[Chapter III, Section 4, Proposition 8]{Se}, 
$\delta(F_n'/F_n)=\delta(F_n'/F_n'')\delta(F_n''/F_n)$. Thus $v(\delta(F_n'/F_n))=v(\delta(F_n'/F_n''))+v(\delta(F_n''/F_n))$ and so 
$v(\delta(F_n'/F_n))\rightarrow 0$ as $n\rightarrow \infty$ implies $v(\delta(F_n''/F_n))\rightarrow 0$ as $n\rightarrow \infty$.
\end{proof}

We now prove Theorem \ref{localfieldDR}. If $K$ is Coates Greenberg deeply ramified then $v$ is not discrete by Lemma 2.12 \cite{CG} and if $K$ is deeply ramified then $v$ is not discrete by Theorem \ref{GRThm2}. Thus we may assume that $v$ is not discrete.

Suppose that $K$ is deeply ramified. We have that  
$$
0=\Omega_{\cO_{\overline{\Q_p}}|\cO_K}=\lim_{\rightarrow}\left(\Omega_{\cO_{K'}|\cO_K}\otimes_{\cO_{K'}}\cO_{\overline{\Q_p}}\right)
$$
where the limit is over the subextensions $K'$ of $\overline{\Q_p}$ such that $K'$ is finite Galois over $K$, so for all such $K'$, we have natural injections of
$\Omega_{\cO_{K'}|\cO_K}$ into $\Omega_{\cO_{\overline{\Q_p}}|\cO_K}$ by Theorem \ref{Theorem4*}, and thus $\Omega_{\cO_{K'}|\cO_K}=0$. By Proposition \ref{CGC}, we have that $v(\delta(F_n'/F_n))\rightarrow 0$ as $n\rightarrow 0$ for all finite Galois subextensions $K'/K$ of $\overline{\Q_p}$. Thus $K$ is Coates Greenberg deeply ramified by Lemma \ref{LemmaCGA}. 

Now suppose that $K$ is Coates Greenberg deeply ramified. Let $K\rightarrow K'$ be a finite Galois extension. Then $\Omega_{\cO_{K'}|\cO_K}$ is almost zero by Proposition \ref{CGC}. Since
$$
\Omega_{\cO_{\overline{\Q_p}}|\cO_K}=\lim_{\rightarrow}\left(\Omega_{\cO_{K'}|\cO_K}\otimes_{\cO_{K'}}\cO_{\overline{\Q_p}}\right)
$$
where the limit is over the subextensions $K'$ of $\overline{\Q_p}$ such that $K'$ is finite Galois over $K$,  we have that 
$\Omega_{\cO_{\overline{\Q_p}}|\cO_K}$ is almost zero. By Theorem \ref{Theorem4'},
 $\Omega_{\cO_{\overline{\Q_p}}|\cO_K}=0$. Thus $K$ is deeply ramified.

\end{document}